\documentclass[a4paper, 11pt]{amsart} 
\usepackage[headings]{fullpage}               
\usepackage{boxedminipage}
\usepackage{amsfonts}
\usepackage{amsmath} 
\usepackage{amssymb}
\usepackage{graphicx}
\usepackage{amsthm}
\usepackage{subfig}
\usepackage{tikz}
\usepackage{setspace}
\usepackage{hyperref}
\hypersetup{
  colorlinks   = true, 
  urlcolor     = blue, 
  linkcolor    = blue, 
  citecolor   = red 
}
\usepackage{enumerate}
\usepackage{xcolor}
\hypersetup{linktocpage=true,}
\newtheorem{theorem}{Theorem}[section]
\newtheorem{lemma}[theorem]{Lemma}
\newtheorem{conjecture}[theorem]{Conjecture}
\newtheorem{proposition}[theorem]{Proposition}
\newtheorem{corollary}[theorem]{Corollary}

\theoremstyle{definition}\newtheorem{definition}{Definition}[section]
\theoremstyle{definition}

\theoremstyle{remark}\newtheorem{remark}{Remark}[section]

\theoremstyle{remark}\newtheorem*{finalremark}{Final remark}

\usepackage{etoolbox}
\let\bbordermatrix\bordermatrix
\patchcmd{\bbordermatrix}{8.75}{4.75}{}{}
\patchcmd{\bbordermatrix}{\left(}{\left[}{}{}
\patchcmd{\bbordermatrix}{\right)}{\right]}{}{}

\usepackage{kbordermatrix} 
\renewcommand{\kbldelim}{[} 
\renewcommand{\kbrdelim}{]}

\newcommand{\rank}{\operatorname{rank}}

\newcommand{\Iso}{\operatorname{Isom}}

\newcommand{\Crit}{\operatorname{Crit}}

\newcommand{\radii}{\operatorname{radii}}

\definecolor{colR}{rgb}{.932,.172,.172}
\definecolor{colB}{rgb}{.255,.41,.884}
\definecolor{colG}{rgb}{0,0.7,0}

\tikzstyle{vertex}=[circle, draw, fill=black, inner sep=0pt, minimum size=4pt]
\tikzstyle{smallvertex}=[circle, line width=1.5pt, draw, fill=black, inner sep=0pt, minimum size=2pt]

\tikzstyle{edge}=[line width=1.5pt]
\tikzstyle{dedge}=[edge,dashed]
\tikzstyle{redge}=[edge,colR]
\tikzstyle{bedge}=[edge,colB]
\tikzstyle{gedge}=[edge,colG]
\tikzstyle{lnode}=[circle,white,draw, fill=black,inner sep=1pt, font=\scriptsize]

\begin{document}

\title{Homothetic packings of centrally symmetric convex bodies}

\author{Sean Dewar}\thanks{Supported by the Austrian Science Fund (FWF): P31888}
\address{Johann Radon Institute for Computational and Applied Mathematics (RICAM), Austrian Academy of Sciences, 4040 Linz, Austria
}
\email{sean.dewar@ricam.oeaw.ac.at}


\begin{abstract}
	A centrally symmetric convex body is a convex compact set with non-empty interior that is symmetric about the origin.
	Of particular interest are those that are both smooth and strictly convex
	-- known here as regular symmetric bodies -- 
	since they retain many of the useful properties of the $d$-dimensional Euclidean ball.
	We prove that for any given regular symmetric body $C$,
	a homothetic packing of copies of $C$ with randomly chosen radii will have a $(2,2)$-sparse planar contact graph.
	We further prove that there exists a comeagre set of centrally symmetric convex bodies $C$ where any $(2,2)$-sparse planar graph can be realised as the contact graph of a stress-free homothetic packing of $C$.
\end{abstract}

\maketitle
\tableofcontents

\section{Introduction}\label{sec:intro}

Given a \emph{convex body} $C$ -- a compact convex subset of $d$-dimensional Euclidean space with non-empty interior --
a \emph{(finite) $C$-packing} is a finite collection $P = \{C_v: v \in V\}$ of homothetic copies of $C$ 
where the interiors of $C_v$ and $C_w$ are disjoint for distinct $v,w \in V$.
We define $G=(V,E)$ to be the \emph{contact graph} of $P$,
with an edge connecting distinct vertices $v,w \in V$ if and only if $C_v \cap C_w \neq \emptyset$.
We define the \emph{placement} $p: V \rightarrow \mathbb{R}^d$ and the \emph{radii} $r: V \rightarrow \mathbb{R}_{>0}$ of $P$ to be the unique maps where $C_v = r_v C +p_v$ for all $v \in V$.
By abuse of notation we write $P=(G,p,r)$ to indicate $P$ is the unique $C$-packing defined by the triple $(G,p,r)$.

The most well-known type of $C$-packing is the special case of $C$ being a Euclidean disc,
often referred to as a \emph{disc packing}.
The well-known Koebe-Andreev-Thurston (KAT) theorem states that every planar graph is the contact graph of some disc packing,
and this disc packing is unique (up to M\"{o}bius transformations and reflections) if and only if the graph is maximal planar \cite{andreev, koebe, thurston}.
As homothetic convex body packings are only a slight generalisation of disc packings,
one would naturally ask whether they have their own version of the KAT theorem.
This was proven to be true by Schramm \cite{schramm90} whenever the convex body is \emph{smooth},
i.e.~every boundary point has a unique supporting hyperplane.

\begin{theorem}\label{t:oded}\cite{schramm90}
	Let $C$ be a smooth convex body in the plane and $G$ a planar graph.
	Then there exists a $C$-packing $P$ with contact graph $G$.
\end{theorem}

The smoothness assumption is required for the theorem to hold;
for instance, if $C$ is a square then no triangulation on more than 4 vertices can be realised as a $C$-packing.
Interestingly,
Schramm also proved in \cite{schramm91} that we must forgo the uniqueness given in the KAT theorem unless we assume the convex body is also \emph{strictly convex},
i.e.~every supporting hyperplane intersects the convex body at exactly one point;
see Theorem \ref{t:odedmax} for the statement of Schramm's extended result for smooth and strictly convex bodies.

Any disc packing can be consider to have a ``sticky'' boundary by assuming that if two discs touch then the contact between them can no longer be broken.
We can then ask whether there exists a motion of the disc packing that maintains the radii of the discs and contacts between them;
we say a disc packing is \emph{sticky flexible} if such a motion exists,
and \emph{sticky rigid} otherwise.
Determining whether a given disc packing is sticky rigid is a geometric constraint problem,
and as such we can use many of the tools utilised in the study of framework rigidity,
including infinitesimal rigidity and independence \cite{asiroth, asiroth2}.
A disc packing $P=(G,p,r)$ is \emph{independent} or \emph{stress-free} if the only function $a :E \rightarrow \mathbb{R}$ that satisfies for each vertex $v \in V$ the equation $\sum_{w \in N(v)} a_{vw}(p_v-p_w)=0$ (with $N(v)$ being the neighbours of $v$) is the zero function.
Independence is an important property for packings and frameworks,
as it allows us to determine the structure of their underlying graphs \cite{laman, geiringer}.
The following result by Connelly, Gortler and Theran proves that all disc packings with generic radii are independent,
allowing us to determine sparsity properties of their contact graphs;
we remember that a graph $G=(V,E)$ is \emph{$(2,k)$-sparse} if $|E'| \leq 2|V'|-k$ for all subgraphs $G'=(V',E')$ with at least one edge,
and \emph{$(2,k)$-tight} if it is $(2,k)$-sparse and $|E|=2|V|-k$ also.

\begin{theorem}\cite{congortheran2019}\label{t:ConGortThur}
	Let $P=(G,p,r)$ be a disc packing.
	If the set $\{r_v:v \in V\}$ is algebraically independent over the rational numbers,
	then $G$ is a $(2,3)$-sparse planar graph.
	Furthermore,
	$G$ is $(2,3)$-tight also if and only if $P$ is sticky rigid.
\end{theorem}

Our question now is the following;
can we state a similar result for any convex body?
To simplify this problem we shall restrict ourselves to \emph{centrally symmetric} convex bodies,
i.e.~if $x \in C$ then $-x \in C$ also.
Since every centrally symmetric convex body $C$ defines a unique norm $\|\cdot\|_C$ via its Minkowski functional,
the restriction allows us to restate the question as a framework rigidity problem with the Euclidean norm replaced with the new norm.
The effects of using different norms in rigidity theory has only been considered in recent years;
we refer the reader to \cite{noneuclidean} for the origin of research on the topic,
\cite{dew1, dew2} for some important classical rigidity theories applied to general norms,
\cite{polyhedra, matrixnorm, maxwell} for results regarding specific norms and symmetry,
and \cite{rotation, meera} for other areas that are closely related.
Most importantly for us,
our definitions of sticky rigidity, sticky flexibility and independence can all be defined in the more general setting (see Sections \ref{sec:prelimrig} and \ref{sec:defnsticky} for explicit definitions).

This generalisation to homothetic centrally symmetric convex body packings will not be enough, however,
as we can still fall prey to degenerate cases (for example, see Proposition \ref{p:counterex}).
We will be required to restrict ourselves to the set of \emph{regular symmetric bodies},
centrally symmetric convex bodies that are smooth and strictly convex.
This is not as strong an assumption as it first seems,
since the set of regular symmetric bodies is a comeagre subset of the set of centrally symmetric convex bodies
(Proposition \ref{p:baire}).
We can also ignore the case when the convex body is a linear transform of a disc,
as this is covered by Theorem \ref{t:ConGortThur};
remember, we are restricting to homothetic packings only.
With these added restrictions, we shall prove the following,
remembering that a property holds for \emph{almost all} elements of a $n$-dimensional Lebesgue measurable set if the set of elements in the set where it does not hold has measure zero.

\begin{theorem}\label{t:body}
	Let $C$ be a regular symmetric body in the plane that is not a linear transform of a disc.
	For almost all $r \in \mathbb{R}^{|V|}_{>0}$,
	the following holds for any $C$-packing $P = (G,p,r)$ with $k := 2|V|-|E|-1$;
	\begin{enumerate}[(i)]
		\item \label{t:body1} $G$ is a $(2,2)$-sparse planar graph,
		\item \label{t:body2} if $\|\cdot\|_C$ is $k$-times continuously differentiable on $\mathbb{R}^2\setminus \{0\}$
		then $P$ is independent, and 
		\item \label{t:body3} if $G$ is $(2,2)$-tight then $P$ is sticky rigid.
	\end{enumerate}
\end{theorem}

It was conjectured in \cite{congortheran2019} that any $(2,3)$-sparse graph is the contact graph of a disc packing with algebraically independent radii;
this is equivalent to the conjecture that any $(2,3)$-sparse graph is the contact graph of an independent disc packing.
We would also conjecture an equivalent result regarding regular symmetric bodies.

\begin{conjecture}\label{conj:generic}
	Let $C$ be a regular symmetric body in the plane that is not a linear transform of a disc and let $G$ be a $(2,2)$-sparse graph.
	Then there exists an independent $C$-packing with contact graph $G$.
\end{conjecture}

Instead of proving Conjecture \ref{conj:generic},
we shall instead take a different tact and prove that the set of centrally symmetric convex bodies where this holds is a comeagre set, and hence the conjecture can be considered to be a ``generic'' property.

\begin{theorem}\label{t:symperfect}
	There exists a comeagre subset $\mathcal{G}$ of the centrally symmetric bodies in the plane where the following holds; 
	for any $C \in \mathcal{G}$ and any $(2,2)$-sparse planar graph $G$,
	there exists an independent $C$-packing with contact graph $G$.
\end{theorem}

The hope for Theorem \ref{t:symperfect} would be to use the idea of approximating discs with other types of convex bodies in some way to prove the conjecture of Connelly, Gortler and Thurston \cite{congortheran2019}.

The layout of the paper shall be as follows.
In Section \ref{sec:prelim} we shall set out the required background material for convex geometry and rigidity theory for general normed spaces, as well as defining most of the notation we shall use for later sections.
In Section \ref{sec:stickyrigid} we shall prove Theorem \ref{t:body};
although many of the ideas implemented were originally set out in \cite{congortheran2019},
there are significant technicalities that first need to be addressed.
In Section \ref{sec:regsymbodies} we shall discuss the space of all centrally symmetric convex bodies, its topology and some important subspaces it has that will be required in later sections.
In Section \ref{sec:maxplanar} we shall prove a generalisation of Theorem \ref{t:oded} that allows us to determine how packings will change as we alter the corresponding convex body;
this will allow us to prove a technical result (Corollary \ref{cor:generalpos}) we will require later.
Finally,
in Section \ref{sec:symperfect} we shall prove Theorem \ref{t:symperfect}.

\section{Preliminaries}\label{sec:prelim}

\subsection{Normed spaces and convex bodies}

Let $\| \cdot \|_C$ be a norm of $\mathbb{R}^d$ with \emph{unit ball} $C = \{ x \in \mathbb{R}^d : \| x \|_C \leq 1\}$.
The unit balls of norms have three important properties:
\begin{enumerate}[(i)]
	\item They are centrally symmetric convex bodies.
	\item Two norms are identical if and only if they have the same unit ball.
	\item Two norms are isometric if and only if one unit ball is the linear transformation of the other;
	this follows from Mazur-Ulam's theorem \cite{mazurulam}.
\end{enumerate}
In fact,
if we have a centrally symmetric convex body $C$ then we can define a norm $\| \cdot \|_C$ with
\begin{align*}
	\|x \|_C := \inf \{ \lambda >0 : x \in \lambda C \}.
\end{align*}
Hence,
there is a one to one mapping between the centrally symmetric convex bodies and the norms of $\mathbb{R}^d$.
Given a centrally symmetric convex body $C$,
we see that
the interior $C^o$ of $C$ is exactly the set of points $x$ where $\|x\|_C <1$,
and the boundary $\partial C := C \setminus C^o$ of $C$  is exactly the set of points $x$ where $\|x\|_C =1$.
For any norm $\| \cdot \|_C$,
we define $\| \cdot \|_C^* := \sup_{y \in C} |\cdot . y|$ to be the \emph{dual norm}.

For $x:= (x(1), \ldots,x(d)) \in \mathbb{R}^d$,
we define $[x]_i := x(i)$.
We shall always reserve $\| \cdot\|$ to be the standard Euclidean norm of $\mathbb{R}^d$,
i.e.~for any $x  \in \mathbb{R}^d$ we will have
\begin{align*}
	\|x\| := \sqrt{ [x]_1^2 + \ldots + [x]_d^2}.
\end{align*}
We shall define $\mathbb{B}^d$ to be the closed unit ball of $(\mathbb{R}^d,\| \cdot \|)$;
in the special case of $d=2$ we shall refer to $\mathbb{D} := \mathbb{B}^2$ as the \emph{unit disc} and $\mathbb{S} := \partial \mathbb{B}^2$ as the \emph{unit circle}.
If a normed space $(\mathbb{R}^d,\| \cdot \|_C)$ is isometric to $(\mathbb{R}^d,\| \cdot \|)$ then we define it to be \emph{Euclidean},
and \emph{non-Euclidean} otherwise.
Equivalently,
$(\mathbb{R}^d,\| \cdot \|_C)$ is Euclidean if and only if $C$ is an ellipsoid (i.e.~there exists a linear transform $T$ so that $T(C)=\mathbb{B}^d$).

A convex body $C$ is \emph{strictly convex} if any tangent hyperplane intersects $C$ only at a single point,
and \emph{smooth} if every point in $\partial C$ has exactly one hyperplane tangent to them.
If $C$ is centrally symmetric then we define $\| \cdot \|_C$ to be \emph{strictly convex} (respectively, \emph{smooth}) if $C$ is strictly convex (respectively, \emph{smooth}).
Equivalently,
we can define $\|\cdot \|_C$ to be strictly convex if $\|x+y\|_C <\|x\|_C + \|y\|_C$ for all linearly independent $x,y$,
and $\|\cdot \|_C$ to be smooth if $\| \cdot \|_C$ is differentiable at every non-zero point.
Any convex body that is centrally symmetric, smooth and strictly convex is called a \emph{regular symmetric body}.

For any point $x \in \mathbb{R}^d$,
we define a point $z \in \mathbb{R}^d$ to be a \emph{support} of $x$ (with respect to a centrally symmetric convex body $C \subset \mathbb{R}^d$) if $z.x=\|x\|_C^2$ and $\|z\|_C^*= \|x\|_C$.
It follows from the Hahn-Banach theorem that every point has a support,
and every point is the support of another point.
For any non-zero point $x \in \mathbb{R}^d$,
we say $x$ \emph{smooth} if it has a unique support
and \emph{exposed} if no point that supports $x$ supports any other point.
Since tangent hyperplanes of $C$ correspond to supports of points in the boundary of $C$,
we have the following;
$C$ is smooth (respectively, strictly convex) if and only if every non-zero point is smooth (respectively, exposed).

It can be shown that a non-zero point $x$ is smooth if and only if $\| \cdot\|_C$ is differentiable at $x$,
and the unique support $x$ (denoted by $\varphi_C(x)$) is exactly the derivative of $\frac{1}{2}\|\cdot\|_C^2$ at $x$;
see \cite[Lemma 1]{maxwell} for a detailed proof.
Using the notation for the unique support of a point,
we define the \emph{duality map} $\varphi_C : x \mapsto \varphi_C(x)$ on the set of smooth points plus the zero point,
with $\varphi_C(0)=0$.

\begin{proposition}\label{p:support}
	For any centrally symmetric convex body $C \subset \mathbb{R}^d$,
	the following holds:
	\begin{enumerate}[(i)]
		\item \label{p:support0} Almost all points of $\mathbb{R}^d$ are smooth.
		\item \label{p:support1} $\varphi_C$ is continuous and homogeneous,
		i.e.~$\varphi_C(\alpha x)= \alpha \varphi_C(x)$
		for all smooth points $x$ and $\alpha \in \mathbb{R}$.
		\item \label{p:support2} $\varphi_C$ is injective if and only if $C$ is strictly convex.
		\item \label{p:support3} $\varphi_C$ is surjective if and only if $C$ is smooth.
		\item \label{p:support4} $\varphi_C$ is a homeomorphism of $\mathbb{R}^d$ to itself if and only if $C$ is a regular symmetric body.
	\end{enumerate}
\end{proposition}

\begin{proof}
	(\ref{p:support0}) \& (\ref{p:support1}): 
	Since the norm $\|\cdot \|_C$ is positively homogeneous,
	the duality map $\varphi_C$ is homogeneous.
	As $\| \cdot \|_C$ is convex,
	both (\ref{p:support0}), (\ref{p:support1}) now hold due to \cite[Theorem 25.5]{rockafellar}.
	
	(\ref{p:support2}):
	See \cite[Proposition 5.4.25]{megginson}.
	
	(\ref{p:support3}):
	The support of $0$ is itself,
	so choose any non-zero $z \in \mathbb{R}^d$.
	By the Hahn-Banach theorem,
	$z$ is the support of some point $x \neq 0$.
	As $x$ is smooth then $\varphi_C(x)=z$.	
	
	(\ref{p:support4}):
	By (\ref{p:support0})--(\ref{p:support3}),
	$C$ is strictly convex and smooth if and only if $\varphi_C$ is a continuous bijective map from $\mathbb{R}^d$ to $\mathbb{R}^d$.
	By Brouwer's theorem for invariance of domain (see \cite[Theorem 1.18]{manifoldlee}),
	$\varphi_C$ is a continuous bijective map if and only if it is a homeomorphism as required.
\end{proof}

For each centrally symmetric convex body $C \subset \mathbb{R}^d$ we define
\begin{align*}
	C^* := \{ x \in \mathbb{R}^d : x.y \leq 1 \text{ for all } y \in C \}.
\end{align*}
It is immediate that $C^{**}=C$,
$C \subset D$ implies $D^* \subset C^*$ and $(\lambda C)^* = \lambda^{-1}C^*$ for all $\lambda >0$.

\begin{proposition}\label{p:polar}
	For any centrally symmetric convex body $C \subset \mathbb{R}^d$,
	the following holds:
	\begin{enumerate}[(i)]
		\item \label{p:polar1} For all $x \in \mathbb{R}^d$,
		$\|x\|_{C^*} = \|x\|_C^*$.
		\item \label{p:polar2} If $C$ a regular symmetric body then $\varphi_{C^*} = \varphi_C^{-1}$.
		\item \label{p:polar3} $C$ is smooth (respectively, strictly convex) if and only if $C^*$ is strictly convex (respectively, smooth).
	\end{enumerate}
\end{proposition}

\begin{proof}
	(\ref{p:polar1}): 
	This holds as
	\begin{align*}
		\|x\|_{C^*} = \inf \{ \lambda >0 : x \in \lambda C^* \} =  \inf \{ \lambda >0 : x.y \leq \lambda \text{ for all } y \in C \} = \sup \{ x.y : y \in C \} = \|x\|_C^*.
	\end{align*}
	
	(\ref{p:polar2}):
	This can be seen by noting that if $y$ is the the support of $x$ with respect to $C$ then $x$ is the support of $y$ with respect to $C^*$.
	
	(\ref{p:polar3}):
	See\cite[Proposition 5.4.7]{megginson}.
\end{proof}

\subsection{Rigidity of frameworks in normed spaces}\label{sec:prelimrig}

A \emph{framework (in $\mathbb{R}^d$)} is a pair $(G,p)$ where $G =(V,E)$ is a (finite simple) graph and $p := (p_v)_{v \in V} \in \mathbb{R}^{d|V|}$;
we define $p$ to be a \emph{placement} of $G$ in $\mathbb{R}^d$.
Given a norm $\| \cdot \|_C$ of $\mathbb{R}^d$ with isometry group $\Iso (\| \cdot \|_C)$,
we define two placements $p,q$ of $G$ to be \emph{congruent} (which we denote by $p \sim q$) if there exists $g \in \Iso(\|\cdot\|_C)$ so that $g(p_v) =q_v$ for all $v \in V$.

The \emph{(squared) rigidity map (of $G$ in $(\mathbb{R}^d,\| \cdot \|_C)$)} is the map
\begin{align*}
	f_{G,C} : \mathbb{R}^{d|V|} \rightarrow \mathbb{R}^{|E|}, ~ (p_v)_{v \in V} \mapsto \left(\frac{1}{2}\|p_v-p_w\|_C^2 \right)_{vw \in E}.
\end{align*}
A framework $(G,p)$ is \emph{well-positioned} (with respect to $\| \cdot \|_C$) if $f_{G,C}(p) \in \mathbb{R}^{|E|}_{>0}$
(i.e.~no edge has length $0$) and $f_{G,C}$ is differentiable at $p$.
If $(G,p)$ is well-positioned then we define the \emph{rigidity matrix (of $(G,p)$ in $(\mathbb{R}^d,\| \cdot \|_C)$)} to be the $|E| \times d|V|$ matrix $R_C(G,p)$ with entries 
\begin{align*}
	R_C(G,p)_{e,(v,i)}:=
	\begin{cases}
		[\varphi_C(p_v-p_w)]_i &\text{if } e = vw, \\
		0 &\text{otherwise}.
	\end{cases}
\end{align*}
As $\varphi_C(x)$ is the derivative of $\frac{1}{2}\|\cdot\|_C$ at $x$,
it follows that $R_C(G,p)$ is the derivative of $f_{G,C}$ at $p$.

Given $\mathcal{T}(\| \cdot \|_C)$ is the tangent space of $\Iso(\| \cdot \|_C)$ at the identity map
and $k$ is the dimension of $\Iso(\| \cdot \|_C)$ 
(i.e.~$k := \dim \mathcal{T}(\| \cdot \|_C)$),
we define the following for any framework $(G,p)$ in a normed space $(\mathbb{R}^d,\| \cdot \|_C)$:
\begin{enumerate}[(i)]
	\item A \emph{flexible motion} of $(G,p)$ is a continuous path $\alpha : [0,\delta] \rightarrow \mathbb{R}^{d|V|}$ ($\delta>0$) where $\alpha(0) = p$ and $f_{G,C}(\alpha(t)) = f_{G,C}(p)$ for all $t \in [0,\delta]$.
	If $\alpha(t) \sim p$ for all $t \in [0,\delta]$ then $\alpha$ is \emph{trivial},
	otherwise $\alpha$ is \emph{non-trivial}.
	\item Given $(G,p)$ is well-positioned,
	we define any element $u \in \mathbb{R}^{d|V|}$ to be an \emph{(infinitesimal) flex} of $(G,p)$ if $u \in \ker R_C(G,p)$.
	If there exists an affine map $g  \in \mathcal{T}(\| \cdot \|_C)$ where $g(p_v)=u_v$ then $u$ is \emph{trivial},
	otherwise it is \emph{non-trivial}.
	\item If every flexible motion of $(G,p)$ is trivial then $(G,p)$ is \emph{rigid} (with respect to $\| \cdot \|_C$),
	otherwise $(G,p)$ is \emph{flexible} (with respect to $\| \cdot \|_C$).
	\item If $(G,p)$ is well-positioned and every flex of $(G,p)$ is trivial then $(G,p)$ is \emph{infinitesimally rigid} (with respect to $\| \cdot \|_C$),
	otherwise $(G,p)$ is \emph{infinitesimally flexible} (with respect to $\| \cdot \|_C$).
	Equivalently,
	$(G,p)$ is infinitesimally rigid if and only if $\ker R_C(G,p) = k$.
	\item If $(G,p)$ is well-positioned and $\rank R_C(G,p) = |E|$ then $(G,p)$ is \emph{independent} (with respect to $\| \cdot \|_C$),
	otherwise $(G,p)$ is \emph{dependent} (with respect to $\| \cdot \|_C$).
	If $(G,p)$ is both infinitesimally rigid and independent then $(G,p)$ is \emph{minimally rigid} (with respect to $\| \cdot \|_C$).
\end{enumerate}

\begin{remark}
	While not immediately obvious,
	the definition for independence stated here is equivalent for frameworks in Euclidean spaces to that given in Section \ref{sec:intro}.
	We can define independence for frameworks in normed spaces in a similar fashion to how it is in the introduction by replacing each case of $(p_v-p_w)$ with the support of $(p_v-p_w)$.
\end{remark}

The following useful result tells us that the two types of rigidity mentioned above are equivalent when a framework $(G,p)$ is \emph{constant} (with respect to $\|\cdot\|_C$),
i.e.~there exists a neighbourhood $U \subset \mathbb{R}^{d|V|}$ of $p$ where for all $q \in U$,
$(G,q)$ is well-positioned and $\rank R_C(G,q) = \rank R_C(G,p)$.

\begin{theorem}\label{t:asiroth}\cite{dew1}
	Let $(G,p)$ be a framework in $\mathbb{R}^d$ that is constant with respect to the norm $\| \cdot \|_C$.
	Then the following are equivalent:
	\begin{enumerate}[(i)]
		\item $(G,p)$ is infinitesimally rigid.
		\item $(G,p)$ is rigid.
	\end{enumerate}
\end{theorem}

We can also characterise which graphs will have an infinitesimally rigid placement in a given normed plane solely from the graph's combinatorics.
We remember that a well-positioned framework $(G,p)$ is \emph{regular} (with respect to the norm $\|\cdot\|_C$) if 
$\rank R_C(G,p) \geq \rank R_C(G,q)$
for any well-positioned placement $q$ of $G$.

\begin{theorem}\label{t:laman}\cite{geiringer}\cite{dew2}
	Let $(G,p)$ be a framework in $\mathbb{R}^2$ that is regular with respect to the norm $\| \cdot \|_C$.
	Let $k =3$ if $\| \cdot \|_C$ is Euclidean and $k =2$ otherwise.
	Then the following holds:
	\begin{enumerate}[(i)]
		\item $(G,p)$ is independent if and only if $G$ is $(2,k)$-sparse.
		\item $(G,p)$ is infinitesimally rigid if and only if $G$ contains a $(2,k)$-tight spanning subgraph.
	\end{enumerate}
\end{theorem}

Using the following result we see how the two previous results are related in smooth normed spaces.

\begin{proposition}\label{p:smoothreg}
	Let $(G,p)$ be a framework in $\mathbb{R}^d$ that is regular with respect to the norm $\| \cdot \|_C$.
	If $C$ is smooth then $(G,p)$ is constant.
\end{proposition}

\begin{proof}
	As $\| \cdot\|_C$ is continuously differentiable on $\mathbb{R}^d\setminus \{0\}$ (Proposition \ref{p:support}(\ref{p:support1})),
	the set of well-positioned placements is an open subset of $\mathbb{R}^{d|V|}$.
	By \cite[Lemma 4.4]{dew1},
	the set of regular placements of $G$ is an open subset of the set of well-positioned placements of $G$,
	hence every regular placement is constant.
\end{proof}

\section{Sticky rigidity for packings with random radii}\label{sec:stickyrigid}

\subsection{Centrally symmetric convex body packings}\label{sec:defnsticky}

Let $C$ be a centrally symmetric convex body in $\mathbb{R}^d$ and $P = \{C_v :v \in V\}$ be the $C$-packing uniquely defined by the triple $(G,p,r)$.
It is immediate that $\|p_v-p_w\|_C \geq r_v + r_w$ for all distinct $v,w \in V$,
with equality if and only if $vw \in E$.
We shall refer to the pair $(G,p)$ as the \emph{contact framework} of $P$.

For a given centrally symmetric convex body $C$ and $C$-packing $P=(G,p,r)$,
we define the following terminology:
\begin{enumerate}[(i)]
	\item $P$ is \emph{well-positioned}/\emph{regular}/\emph{constant} if $(G,p)$ is well-positioned/regular/constant with respect to $\|\cdot\|_C$.
	\item $P$ is \emph{sticky rigid} if $(G,p)$ is rigid with respect to $\|\cdot\|_C$,
	otherwise $P$ is \emph{sticky flexible}.
	\item Given $P$ is well-positioned,
	$P$ is \emph{sticky infinitesimally rigid} if $(G,p)$ is infinitesimally rigid with respect to $\|\cdot\|_C$,
	otherwise $P$ is \emph{sticky infinitesimally flexible}.
	\item Given $P$ is well-positioned,
	$P$ is \emph{independent} if $(G,p)$ is independent with respect to $\|\cdot\|_C$,
	otherwise $P$ is \emph{dependent}.
\end{enumerate}

Given a centrally symmetric body $C$,
a graph $G=(V,E)$ and $r \in \mathbb{R}^{|V|}_{>0}$,
we define 
\begin{eqnarray*}
	S_{G,C} &:=& \left\{ (q,s) \in \mathbb{R}^{d|V|} \times \mathbb{R}^{|V|}_{>0}: (G,q,s) \text{ is a $C$-packing} \right\}, \\
	S_{G,C}(r) &:=& \left\{ p \in \mathbb{R}^{d|V|}: (p,r) \in S_{G,C} \right\}.
\end{eqnarray*}
We immediately notice that $P$ is sticky rigid if and only if there is no continuous path in the quotient space $S_{G,C}(r)/\Iso(\|\cdot\|_C)$ that passes through $p$.
Furthermore,
by Proposition \ref{p:smoothreg} and Theorem \ref{t:asiroth} we have the immediate following result.

\begin{proposition}\label{p:infsr}
	Let $C$ be a smooth centrally symmetric convex body.
	Then any infinitesimally rigid $C$-packing is rigid.
\end{proposition}

\subsection{Packing rigidity maps}\label{sec:packmaps}

For any normed space $(\mathbb{R}^d, \|\cdot\|_C)$ and graph $G$,
we define
\begin{align*}
	h_{G,C} : \mathbb{R}^{d|V|} \times \mathbb{R}_{>0}^{|V|} \rightarrow \mathbb{R}^{|E|}, ~ (p,r) \mapsto \frac{1}{2}\left( \|p_v-p_w\|_C^2 - (r_v+ r_w)^2 \right)_{vw\in E}
\end{align*} 
to be the \emph{packing rigidity map of $G$}.
A triple $(G,p,r)$ is a $C$-packing if and only if $h_{G,C}(p,r)=0$ and $h_{G',C}(p,r) \in \mathbb{R}^{|E'|}_{\geq 0} \setminus \{0\}$ for all graphs $G'=(V,E')$ where $E \subsetneq E'$.
It follows that the set $S_{G,C}$ is an open subset of $h_{G,C}^{-1}[0]$.

Let $(G,p)$ be a well-positioned framework with respect to a norm $\|\cdot\|_C$ and $r \in \mathbb{R}_{>0}^{|V|}$.
Define $I(G,r)$ to be the $|E| \times |V|$ matrix with entries 
\begin{align*}
	I(G,r)_{e,v}:=
	\begin{cases}
		-(r_v+r_w) &\text{if } e = vw, \\
		0 &\text{otherwise}.
	\end{cases}
\end{align*}
By adding the columns of $I(G,r)$ to $R_C(G,p)$ we form the \emph{packing rigidity matrix of $(G,p,r)$},
the matrix $R_C(G,p,r) := [ R_C(G,p) \quad I(G,r)]$ with $|E|$ rows and $(d+1)|V|$ columns.
We define any column corresponding to $R_C(G,p)$ to be a \emph{point column} and any column corresponding to $I(G,r)$ to be a \emph{radii column}.
Each row $vw$ of $R_C(G,p,r)$ is of the form
\begin{align*}
	\kbordermatrix{
	 &(v,i) \text{ point}& &(w,i) \text{ point}& &v, \text{ radii}& & w, \text{ radii}\\
	 & [\varphi_C(p_v-p_w)]_i & \cdots & [\varphi_C(p_w-p_v)]_i & \cdots & -(r_v+r_w) & \cdots &  -(r_v+r_w)\\ 
	},
\end{align*}
with $0$ at all other values.
The packing rigidity matrix $R_C(G,p,r)$ is the derivative of the packing rigidity map $h_{G,C}$ at the point $(p,r)$.
If $(G',p,r)$ is a well-positioned $C$-packing with $G'=(V,E')$ containing $G$ then $r_v+r_w = \|p_v-p_w\|_C$ for all $vw \in E$;
by substituting this equality into $R_C(G,p,r)$ we obtain the matrix with rows
\begin{align}\label{e:edgestress}
	\kbordermatrix{
	 &(v,i) \text{ point}& &(w,i) \text{ point}& &v, \text{ radii}& & w, \text{ radii}\\
	 & [\varphi_C(p_v-p_w)]_i & \cdots & [\varphi_C(p_w-p_v)]_i & \cdots & -\|p_v-p_w\|_C & \cdots & -\|p_v-p_w\|_C\\ 
	}.
\end{align}

\subsection{Edge-length equilibrium stresses}\label{sec:edgestress}

We begin with the following definition,
where we remember that $N(v)$ is the neighbourhood of a vertex $v$.

\begin{definition}
	Let $(G,p)$ be a well-positioned framework in a normed space $(\mathbb{R}^d,\|\cdot\|_C)$.
	A map $a : E \rightarrow \mathbb{R}$ is an \emph{edge-length equilibrium stress} if $a_{vw}\neq 0$ for some $vw \in E$ and 
\begin{align*}
	\sum_{w \in N(v)} a_{vw} \varphi_C(p_v-p_w) = 0, \qquad \sum_{w \in N(v)} a_{vw} \|p_v-p_w\|_C = 0
\end{align*}
for all $v \in V$.
\end{definition}

If $P=(G,p,r)$ is a well-positioned $C$-packing then, following from (\ref{e:edgestress}),
the matrix $R_C(G,p,r)$ has independent rows if and only if the framework $(G,p)$ has no edge-length equilibrium stresses.
In fact, no homothetic regular symmetric body packing in the plane has any edge-length stress.

\begin{lemma}\label{lem:2.10}
	Let $C$ be a regular symmetric body in the plane and $P=(G,p,r)$ be a $C$-packing.
	Then $(G,p)$ has no edge-length equilibrium stresses.
\end{lemma}

In order to prove Lemma \ref{lem:2.10} we shall first define the following.
Let $(G,p)$ be a \emph{planar framework};
a framework $(G,p)$ in $\mathbb{R}^2$ with $p_v\neq p_w$ for all distinct vertices $v,w\in V$,
where if we consider each edge $vw$ as the line segment $[p_v,p_w]$, 
the edges only intersect each other at their ends.
Choose any function $a : E \rightarrow \mathbb{R}$,
and define a vertex $v \in V$ to be \emph{relevant} if $a(vw) \neq 0$ for some $w \in N(v)$.
For any relevant vertex $v \in V$,
the set
\begin{align*}
	N(v)^a_p := \left\{ \frac{p_w - p_{v}}{\|p_w - p_{v}\|} : w \in N(v), a(vw) \neq 0 \right\}
\end{align*}
is non-empty subset of the circle $\mathbb{S}$.
By starting at a point of the circle that is not in $N(v)^a_p$,
we may define an order $v_1, \ldots, v_n$ of the vertices $w$ of $N(v)$ with $a(vw) \neq 0$ as we travel clockwise around the circle and back to the starting point,
i.e.,
we first reach $\frac{p_{v_1} - p_{v}}{\|p_{v_1} - p_{v}\|}$,
then $\frac{p_{v_2} - p_{v}}{\|p_{v_2} - p_{v}\|}$,
etc.
We define $I_v$ (the \emph{index} of $v$ with respect to $a$) to be the number of times the sign changes in the sequence $a_{v_1 v} , a_{v_2 v}, \ldots, a_{v_n v} , a_{v_1 v}$.
It is immediate that $I_v$ is invariant under changing our starting point on the circle or changing the direction we traverse around the circle.
See Figure \ref{fig:relvert} for an example of a relevant vertex and its index.

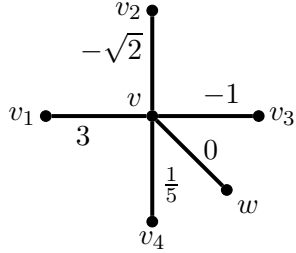
\begin{figure}[ht]
	\begin{tikzpicture}[scale=0.7]
		\node[vertex] (1) at (-2,0) {};
		\node[vertex] (2) at (0,2) {};
		\node[vertex] (3) at (2,0) {};
		\node[vertex] (4) at (1.4,-1.4) {};
		\node[vertex] (5) at (0,-2) {};
		\node[vertex] (0) at (0,0) {};
		
		\node [left] at (-2,0) {$v_1$};
		\node [left] at (0,2) {$v_2$};
		\node [right] at (2,0) {$v_3$};
		\node [below] at (0,-2) {$v_4$};
		\node [below right] at (1.4,-1.4) {$w$};
		\node [above left] at (0,0) {$v$};
		
		\node [below] at (-1.3,0) {$3$};
		\node [left] at (0,1.3) {$-\sqrt{2}$};		
		\node [above] at (1.3,0) {$-1$};
		\node [right] at (0,-1.3) {$\frac{1}{5}$};
		\node [above] at (1.1,-1) {$0$};
		
		\draw[edge] (0)edge(1);
		\draw[edge] (0)edge(2);
		\draw[edge] (0)edge(3);
		\draw[edge] (0)edge(4);		
		\draw[edge] (0)edge(5);	
	\end{tikzpicture}
	\caption{A relevant vertex $v$ and its neighbours $v_1,v_2,v_3,v_4,w$;
	the map $a:E \rightarrow \mathbb{R}$ is represented by the numbers next to each edge.
	We order the neighbours $v_1,v_2,v_3,v_4$, ignoring $w$ as $a_{vw}=0$. 
	We note that $I_v = 2$, which is independent of where our ordering begins and whether we travel around the neighbours of $v$ clockwise or anti-clockwise.}
	\label{fig:relvert}
\end{figure}

The sum of indices of any edge function can be bounded by the amount of relevant vertices by the following result.

\begin{lemma}\label{lem:2.7}\cite[Lemma 5.2]{Glu}
	Let $(G,p)$ be a planar framework. 
	For any function $a : E \rightarrow \mathbb{R}$ we have the inequality $\sum_{v \in V'} I_v \leq 4|V'| - 8$,
	where $V' \subset V$ is the set of relevant vertices.
\end{lemma}

For edge-length equilibrium stresses we can also give a lower bound.
	
\begin{lemma}\label{lem:2.8.1}
	Let $(G,p)$ be a well-positioned planar framework in a normed plane $(\mathbb{R}^2, \|\cdot\|_C)$.
	If $a :E \rightarrow \mathbb{R}$ is an edge-length equilibrium stress of $(G,p)$
	then $I_{v} > 0$ for all relevant vertices $v \in V$.
\end{lemma}	

\begin{proof}
	Suppose $v_0$ is a relevant vertex with $N'(v_0) := \{ v \in N(v) : a_{v v_0} \neq 0 \}$ and $I_{v_0}=0$.
	Without loss of generality we may assume $a_{v v_0}>0$ for all $v \in N'(v_0)$.
	Then 
	\begin{align*}
		\sum_{v \in N(v_0)} a_{v v_0} \|p_v-p_{v_0}\|_C > 0,
	\end{align*}
	which contradicts that $a$ is an edge-length equilibrium stress.
\end{proof}

\begin{lemma}\label{l:conestuff}
	Let $C$ be a regular symmetric body in the plane and $S, T$ be two finite sets of points in $\partial C$.
	Suppose that there exists a non-negative cone $X := \{ a x_1 + bx_2 : a,b \geq 0\}$ for some linearly independent $x_1,x_2 \in \mathbb{R}^2$ where $S \subset X^\circ$ and $T \cap X = \emptyset$.
	Then there exists $y \in \mathbb{R}^2$ where $\varphi_C(x).y >1$ for all $x \in S$ and $\varphi_C(x).y<1$ for all $x \in T$.
\end{lemma}

\begin{proof}
	We may assume that $\|x_1\|_C=\|x_2\|_C=1$.
	As $\varphi_C$ is a homogeneous homeomorphism (Proposition \ref{p:support}),
	it maps non-negative cones to non-negative cones.
	Hence $\varphi_C(X) = Y$ for some non-negative cone $Y= \{ a y_1 + by_2 : a,b \geq 0\}$ with $y_1=\varphi_C(x_1)$ and $y_2=\varphi_C(x_2)$ linearly independent,
	$\varphi_C(S) \subset Y^\circ$, $\varphi_C(T) \cap Y = \emptyset$ and $\varphi_C(S \cup T) \subset \partial C^*$.
	We now choose $y \in \mathbb{R}^2$ to be the unique point where $y.y_1=y.y_2=1$.
	
	Suppose $\varphi_C(x)=ay_1 +by_2 \in \varphi_C(S)$.
	As $C^*$ is strictly convex (Proposition \ref{p:polar}) and $a,b>0$,
	then
	\begin{align*}
		1=\|\varphi_C(x)\|_{C^*} < a\|y_1\|_{C^*} + b\|y_2\|_{C^*} = a+b.
	\end{align*}
	Hence as $y.y_1=y.y_2=1$ then $\varphi_C(x).y =a+b > 1$.
	Now suppose $\varphi_C(x)=ay_1 +by_2 \in \varphi_C(T)$.
	Since $\varphi_C(T) \cap Y = \emptyset$,
	either $a<0$ or $b<0$.
	If they are both negative then clearly $\varphi_C(x).y<0$,
	so without loss of generality assume that $a>0>b$.
	As $C^*$ is strictly convex,
	then
	\begin{align*}
		1=\|\varphi_C(x)\|_{C^*} > |a|\|y_1\|_{C^*} - |b|\|y_2\|_{C^*} = a + b.
	\end{align*}
	It now follows that $\varphi_C(x).y =a + b < 1$ as required.
\end{proof}

For a planar framework $(G,p)$ with map $a : E \rightarrow \mathbb{R}$ and relevant vertex $v_0$,
we shall now define $N^+(v_0) := \{ v \in N(v_0) : a_{v v_0} >0 \}$ and $N^-(v_0) := \{ v \in N(v_0) : a_{v v_0} <0 \}$.
	
\begin{lemma}\label{lem:order}
	Let $(G,p)$ be a planar framework in a smooth and strictly convex normed plane $(\mathbb{R}^2,\|\cdot\|_C)$.
	Suppose $a : E \rightarrow \mathbb{R}$ is a function where $I_{v_0} =2$ for a vertex $v_0 \in V$.
	Then there exists $y \in \mathbb{R}^2$ such that either $\varphi_C(p_v-p_{v_0}).y  > \|p_v - p_{v_0}\|_C$ for all $v \in N^+(v_0)$ and $\varphi_C(p_v-p_{v_0}).y  < \|p_v - p_{v_0}\|_C$ for all $v \in N^-(v_0)$, or vice versa.
\end{lemma}

\begin{proof}
	Define $A^+ := \{(p_v-p_{v_0})/\|p_v-p_{v_0}\|_C : v \in N^+(v_0)\}$ and $A^- := \{(p_v-p_{v_0})/\|p_v-p_{v_0}\|_C : v \in N^-(v_0)\}$.
	As $I_{v_0} =2$ then one of $A^+,A^-$ must lie within a non-negative cone;
	we shall suppose without loss of generality that $A^+$ does.
	By Lemma \ref{l:conestuff},
	there exists $y \in \mathbb{R}^2$ such that 
	\begin{align*}
		\varphi_C\left( \frac{p_v-p_{v_0}}{\|p_v - p_{v_0}\|_C} \right).y  > 1 \text{ for all $v \in N^+(v_0)$}, \qquad \varphi_C\left( \frac{p_v-p_{v_0}}{\|p_v - p_{v_0}\|_C} \right).y <1 \text{ for all $v \in N^-(v_0)$};
		\end{align*}
	we now use that $\varphi_C$ is homogeneous to obtain the required result.
\end{proof}

\begin{lemma}\label{lem:2.8}
	Let $(G,p)$ be a planar framework in a smooth and strictly convex normed plane $(\mathbb{R}^2,\|\cdot\|_C)$.
	If $a$ is an edge-length equilibrium stress of $(G,p)$ and $v_0 \in V$ is a relevant vertex,
	then $I_{v_0} \geq 4$.
\end{lemma}

\begin{proof}
	Suppose $I_{v_0} < 4$.
	As $I_{v_0}$ must be even and $I_{v_0} >0$ by Lemma \ref{lem:2.8.1},
	then $I_{v_0} = 2$.
	By Lemma \ref{lem:order} and multiplying $a$ by $-1$ if necessary,
	there exists $y \in \mathbb{R}^2$ where for all $v \in N^+(v_0)$ and $w \in N^-(v_0)$ we have
	\begin{equation}\label{e:plus}
		\varphi_C(p_v-p_{v_0}).y >  \|p_{v} - p_{v_0}\|_C, \qquad \varphi_C(p_w-p_{v_0}).y <  \|p_{w} - p_{v_0}\|_C.
	\end{equation}
	As $a$ is an edge-length equilibrium stress then
	\begin{equation}\label{e:radii}
		\sum_{v \in N(v_0)} a_{v_0 v} \varphi_C(p_{v} - p_{v_0}) = 0 , \qquad \sum_{v \in N(v_0)} a_{v_0 v} \|p_{v} - p_{v_0}\|_C = 0.
	\end{equation}	
	By combining equations (\ref{e:plus}) and (\ref{e:radii}) it follows that
	\begin{eqnarray*}
		\sum_{v \in N^+(v_0)}   a_{v_0 v} \|p_{v} - p_{v_0}\|_C &=& -\sum_{v \in N^-(v_0)} a_{v_0 v} \|p_{v} - p_{v_0}\|_C 
		\quad > \quad -\sum_{v \in N^-(v_0)} a_{v_0 v} \varphi_C(p_v-p_{v_0}).y \\
		&=& \sum_{v \in N^+(v_0)}  a_{v_0 v} \varphi_C(p_v-p_{v_0}).y 
		\quad > \quad \sum_{v \in N^+(v_0)}   a_{v_0 v} \|p_{v} - p_{v_0}\|_C,
	\end{eqnarray*}
	a contradiction.
\end{proof}

By combining our previous results, we have the following.

\begin{lemma}\label{l:planarnoestress}
	Let $(G,p)$ be a planar framework in a smooth and strictly convex normed plane $(\mathbb{R}^2,\|\cdot\|_C)$.
	Then $(G,p)$ has no edge-length equilibrium stresses.
\end{lemma}

\begin{proof}
	Suppose $(G,p)$ has an edge-length equilibrium stress $a$ and let $V'\subset V$ be the set of relevant vertices of $G$.
	By Lemma \ref{lem:2.7},
	we have $\sum_{v \in V'} I_v \leq 4|V'|-8$ for the indexing $I$ generated by $a$.
	However, by Lemma \ref{lem:2.8} we have $\sum_{v \in V'} I_v \geq 4|V'|$, a contradiction.
\end{proof}

The final ingredient we require is that the contact frameworks of regular symmetric body packings are planar.
To do so we first need the following result of Danzer and Gr\"{u}nbaum.

\begin{theorem}\label{t:DanzerGruenbaum}\cite[II b$\alpha$]{DanzerGruenbaum}
	Let $C$ be a convex body in $\mathbb{R}^d$ and $P=(G,p,r)$ be a $C$-packing with $|V|=n$ and $r_v=1$ for all $v \in V$.
	If $G$ is a complete graph then $n \leq 2^d$;
	further, if $n= 2^d$ then $C$ is a parallelotope.
\end{theorem}

\begin{lemma}\label{l:planar}
	Let $C$ be a centrally symmetric convex body in $\mathbb{R}^2$ and $P=(G,p,r)$ be a $C$-packing.
	If $C$ is not a parallelogram,
	then $(G,p)$ is planar.
\end{lemma}

\begin{proof}
	Suppose $(G,p)$ is not planar.
	By translating $(G,p)$ we may assume the line segments $[p_{v_1},p_{v_2}]$, $[p_{v_3},p_{v_4}]$ cross at the point $0$ for some edges $v_1v_2,v_3v_4 \in E(G)$.
	We note that the line $[p_{v_1},p_{v_2}]$ is covered by $C_{v_1} \cup C_{v_2}$,
	hence $0 \in C_{v_1} \cup C_{v_2}$;
	similarly, $0 \in C_{v_3} \cup C_{v_4}$.
	If $0 \in C_{v_1}^o$ then it follows that the interiors of at least one of $C_{v_3}$ and $C_{v_4}$ must intersect with the interior of $C_{v_1}$,
	contradicting that $P$ is a $C$-packing.
	Since this will also be true for $v_2,v_3,v_4$,
	then $0 \in \bigcap_{i=1}^4 \partial C_{v_i}$.
	Hence the subgraph $G'$ of $G$ generated by $\{v_1,v_2,v_3,v_4\}$ is a complete graph.
	
	We note that as $C$ is convex then for any scalars $c_{v_1},c_{v_2},c_{v_3}, c_{v_4}>0$ we have that $0 \in \partial c_i C_{v_i}$ and $c_iC_{v_i}^\circ \cap c_j C_{v_j}^\circ = \emptyset$ for all distinct $1\leq i,j \leq 4$
	(to see this, first consider the case when the $c_i$'s are all less than $1$).
	If we define $P'=(G',p',r')$ with $p'_{v_i} := p_{v_i}/r_{v_i}$ and $r'_{v_i}=1$ for each $1\leq i\leq 4$,
	then $P'$ is a $C$-packing,
	and by Theorem \ref{t:DanzerGruenbaum},
	$C$ is a parallelogram.
\end{proof}

We are now able to prove the key lemma of the section.

\begin{proof}[Proof of Lemma \ref{lem:2.10}]
	As $C$ is a regular symmetric body then by Lemma \ref{l:planar},
	$(G,p)$ is planar.
	Hence by Lemma \ref{l:planarnoestress},
	$(G,p)$ has no edge-length equilibrium stresses.
\end{proof}

Lemma \ref{lem:2.10} does not hold for all packings of centrally symmetric bodies.

\begin{proposition}\label{p:counterex}
	Let $C$ be a centrally symmetric convex body in the plane that is not strictly convex.
	Then there exists a well-positioned $C$-packing $P=(G,p,r)$ where $(G,p)$ has an edge-length equilibrium stresses.
\end{proposition}

\begin{proof}
	As $C$ is not strictly convex,
	there exists a line segment $[x_1,x_2] \subset \partial C$.
	Choose large enough $t >0$ such that $C + t(x_1-x_2)$ and $C- t(x_1-x_2)$ are disjoint.
	Define $G=(V,E)$ to be the graph with $V := \{v_1,v_2,v_3, v_4\}$ and $E:= \{ v_1 v_2, v_2 v_3, v_3 v_4, v_4v_1\}$.
	Define $P=(G,p,r)$ to be the well-positioned $C$-packing with
	\begin{align*}
		p_{v_1} = (1+t)(x_1 + x_2), \quad & p_{v_2} = t(x_1-x_2), \quad & p_{v_3} = -(1+t)(x_1 + x_2), \quad & p_{v_4} = -t(x_1-x_2) \\
		r_{v_1} = 1+2t, \quad & r_{v_2} = 1, \quad  & r_{v_3} = 1+2t, \quad & r_{v_4} = 1; &
	\end{align*}
	see Figure \ref{fig:sq} for an illustration of the described packing.
	To see that $P$ is a well-positioned $C$-packing we first note that $\varphi_C(ax_1 +bx_2)= (a+b)\varphi_C(x)$ for all $a,b \geq 0$ and $x \in [x_1,x_2]$,
	hence $\|x_1+x_2\|_C = \|x_1\|_C + \|x_2\|_C$ and
	\begin{align*}
		\varphi_C(p_{v_1} - p_{v_2}) = \varphi_C(p_{v_2} - p_{v_3}) =\varphi_C(p_{v_1} - p_{v_4}) = \varphi_C(p_{v_4} - p_{v_3}) = (2+2t)\varphi_C(x_1).
	\end{align*}
	We now note that $(G,p)$ has an edge-length equilibrium stress,
	namely
	\begin{align*}
		a_{v_1v_2} = 1, \quad a_{v_2v_3} = 1, \quad a_{v_3v_4} = -1, \quad a_{v_4v_1} = -1.
	\end{align*}
\end{proof}
%
%
%
%
%

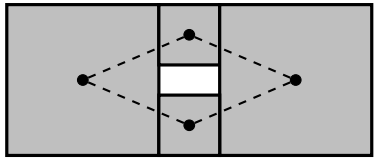
\begin{figure}[ht]
	\begin{tikzpicture}[scale=0.4]
		
		
		\draw[very thick, fill=lightgray] (-6,-2.5) rectangle (-1,2.5);
		\draw[very thick, fill=lightgray] (-1,0.5) rectangle (1,2.5);
		\draw[very thick, fill=lightgray] (1,-2.5) rectangle (6,2.5);
		\draw[very thick, fill=lightgray] (-1,-2.5) rectangle (1,-0.5);

		\node[vertex] (1) at (-3.5,0) {};
		\node[vertex] (2) at (0,1.5) {};
		\node[vertex] (3) at (3.5,0) {};
		\node[vertex] (4) at (0,-1.5) {};
		
		\draw[thick, dashed] (1)edge(2);
		\draw[thick, dashed] (2)edge(3);
		\draw[thick, dashed] (3)edge(4);
		\draw[thick, dashed] (4)edge(1);	
	\end{tikzpicture}
	\caption{A $C$-packing for $C:= [-1,1]^2$ with an edge-length equilibrium stress.}
	\label{fig:sq}
\end{figure}

\subsection{Proof of Theorem \ref{t:body}}

From Lemma \ref{lem:2.10} we obtain the following immediate result.

\begin{lemma}\label{lem:2.10a}
	Let $C$ be a regular symmetric body in the plane and $P=(G,p,r)$ be a $C$-packing.
	Then $R_C(G,p,r)$ is surjective with rank $|E|$.
\end{lemma}

It now follows that if $C$ is a regular symmetric body then the set $S_{G,C}$ of all $C$-packings with contact graph $G$ will have a differentiable manifold structure.

\begin{lemma}\label{lem:2.11}
	Let $C \subset \mathbb{R}^2$ be a regular symmetric body where $\|\cdot\|_C$ is $C^k$-differentiable on the set $\mathbb{R}^2\setminus \{0\}$.
	Then for any graph $G = (V,E)$,
	either $G$ is not planar and $S_{G,C}=\emptyset$,
	or $G$ is planar and $S_{G,C}$ is a $(3|V|-|E|)$-dimensional $C^k$-differentiable manifold with tangent space $\ker R_C(G,p,r)$ at $(p,r)$.
\end{lemma}

\begin{proof}
	By Theorem \ref{t:oded},
	$S_{G,C}$ is empty if and only $G$ is not planar.
	Assume $G$ is planar so that $S_{G,C} \neq \emptyset$.
	As $\| \cdot\|_C$ is $C^k$-differentiable on the non-zero points,
	the map $h_{G,C}$ is $C^k$-differentiable on an open neighbourhood of $S_{G,C}$.
	By Lemma \ref{lem:2.10a},
	$R_C(G,p,r)$ is surjective at every point $(p,r)$ of $S_{G,C}$;
	hence, by the continuity of the derivative of $h_{G,C}$,
	the matrix $R_C(G,p,r)$ is surjective at every point $(p,r)$ in an open neighbourhood of $S_{G,C}$.
	The result now holds by applying the constant rank theorem (\cite[Theorem 2.5.15]{manifold}).
\end{proof}

Since our maps and manifolds will usually not be infinitely differentiable,
we shall require the more general version of Sard's theorem.

\begin{theorem}(see \cite[Theorem 3.6.3]{manifold})\label{t:sard}
	Let $f:M \rightarrow N$ be a $C^k$-differentiable map between $C^k$-differentiable manifolds for $k \geq 1$.
	Let $K \subset N$ be the set of critical values of $f$,
	i.e.~the points $y \in N$ where there exists $x \in M$ with $f(x)=y$ and $\rank df(x) < \dim N$.
	If $k \geq \dim M - \dim N +1$ then $K$ has Lebesgue measure zero.
\end{theorem}

\begin{lemma}\label{lem:2.12}
	Let $G=(V,E)$ be a planar graph and $C \subset \mathbb{R}^2$ be a regular symmetric body where $\|\cdot\|_C$ is $C^k$-differentiable on the set $\mathbb{R}^2\setminus \{0\}$.
	Let $\Crit_G \subset \mathbb{R}_{>0}^{|V|}$ be the set of critical values of the map
	\begin{align*}
		\rho : S_{G,C} \rightarrow \mathbb{R}_{>0}^{|V|}, ~ (p,r) \mapsto r.
	\end{align*}
	If $k \geq 2|V|-|E| -1$ then $\Crit_G$ has Lebesgue measure zero.
	Further,
	if $r \in \rho(S_{G,C}) \setminus \Crit_G$ then $S_{G,C}(r)$ is a $(2|V|-|E|)$-dimensional $C^k$-differentiable manifold.
\end{lemma}

\begin{proof}
	Choose any $v_0 \in V$.
	Define $S_{G,C}^{v_0}$ to be the set of all $C$-packings with contact graph $G$ and $p_{v_0}=0$,
	and $\tilde{\rho}$ to be the restriction of $\rho$ to $S_{G,C}^{v_0}$.
	As $\rho$ is invariant under translation then the critical values of $\tilde{\rho}$ are exactly the critical values of $\rho$.
	The map $\tilde{\rho}$ is $C^k$-differentiable as it is the restriction of a linear map and $S_{G,C}^{v_0}$ is a $(3|V|-|E|-2)$-dimensional $C^k$-differentiable manifold (Lemma \ref{lem:2.11}).
	As
	\begin{align*}
		k  \geq 2|V|-|E| -1 = (3|V|-|E|-2) - |V| + 1
	\end{align*}
	then by Theorem \ref{t:sard},
	$\Crit_G$ has Lebesgue measure zero.
	
	Choose $r \in \rho(S_{G,C}) \setminus \Crit_G$.
	By the preimage theorem (see \cite[Theorem 3.5.4]{manifold}),
	$S_{G,C}(r) = \rho^{-1}[r]$ is a $C^k$-differentiable manifold with dimension $3|V|-|E| - |V| = 2|V|-|E|$.
\end{proof}

The projection map $\rho$ is also very closely related to the rigidity matrix.

\begin{lemma}\label{l:ranknull}
	Let $C \subset \mathbb{R}^2$ be a regular symmetric body and $P=(G,p,r)$ a $C$-packing.
	Then
	\begin{align*}
		\rank d \rho (p,r) = |V| \quad \Leftrightarrow \quad \rank R_C(G,p) = |E|.
	\end{align*}	
\end{lemma}

\begin{proof}
	By Lemma \ref{lem:2.11},
	$S_{G,C}$ is a $C^1$-differentiable manifold with dimension $3|V|-|E|$ and tangent space $\ker R_C(G,p,r)$ at $(p,r)$,
	hence $\rho$ is $C^1$-differentiable.
	We now observe that
	\begin{align*}
		\ker d\rho(p,r) =\left\{ (u,s) \in \ker R_C(G,p,r) : s = 0 \right\} = \ker R_C(G,p) \times \{0\}.
	\end{align*}
	By applying the rank nullity theorem to both $d\rho(p,r)$ and $R_C(G,p)$,
	we see that
	\begin{align*}
		\rank d \rho(p,r) -|V|=2|V|- |E| - \dim \ker d\rho(p,r) = \rank R_C(G,p) - |E|.
	\end{align*}
\end{proof}

Using this result coupled with Lemma \ref{lem:2.12}, we are now able to prove our first main result of the paper.

\begin{proof}[Proof of Theorem \ref{t:body}]
	Let $\|\cdot\|_C$ be $C^\ell$-differentiable on $\mathbb{R}^2\setminus \{0\}$.
	As $C$ is smooth then $\ell \geq 1$ by Proposition \ref{p:support}.
	Define $A$ to be the set of connected graphs $H$ where $|E(H)|\geq 2|V(H)|- 1- \ell$ and $V(H) \subset V$.
	By Lemma \ref{lem:2.12},
	$\Crit_H$ has Lebesgue measure zero for any graph $H \in A$,
	as $\ell \geq 2|V(H)| - |E(H)| - 1$.
	Hence the set 
	\begin{align*}
		\Crit := \bigcup \left\{ \Crit_H \times \mathbb{R}^{|V \setminus V(H)|} : H \in A \right\}
	\end{align*}
	also has Lebesgue measure zero.
	Fix $r \in \mathbb{R}_{>0}^V \setminus \Crit$,
	let $P=(G,p,r)$ be any $C$-packing and set $k:= 2|V|-|E|-1$.
	Without loss of generality,
	we may assume $G$ is connected.
	We shall now show that properties (\ref{t:body1}), (\ref{t:body2}) and (\ref{t:body3}) hold for $P$,
	and hence as we chose $r$ arbitrarily the required result will hold.	
	
	(\ref{t:body1}):
	Suppose $G$ is not $(2,2)$-sparse,
	i.e.~$G$ contains a subgraph $G'=(V',E')$ where $n =2|V'|-|E'|< 2$.
	We may suppose $G'$ is connected,
	hence $G' \in A$.
	By Lemma \ref{lem:2.12},
	the set $S_{G',C}(r|_{V'})$ is a $n$-dimensional $C^\ell$-differentiable manifold for $n \in \{0,1\}$.
	However this contradicts that $S_{G',C}(r|_{V'})$ is invariant under translation,
	as $\dim S_{G',C}(r|{V'}) \geq 2$.
	Hence $G$ is $(2,2)$-sparse.
	
	(\ref{t:body2}):	
	Suppose $\ell \geq k$.
	Since $r \in \mathbb{R}_{>0}^V \setminus \Crit$ then $\rank d\rho(p,r) = |V|$.
	By Lemma \ref{l:ranknull},
	$\rank R_C(G,p) = |E|$ and $P$ is independent.
	
	(\ref{t:body3}):
	Suppose $G$ is $(2,2)$-tight.
	As $C$ is smooth we have $\ell \geq k =1$,
	hence $P$ is independent by (\ref{t:body2}).
	As $\|\cdot\|_C$ is non-Euclidean,
	then by Theorem \ref{t:laman},
	$P$ is infinitesimally rigid.
	Since $C$ is smooth,
	then by Proposition \ref{p:infsr},
	$P$ is sticky rigid.
\end{proof}

\subsection{Sets of possible radii}

For any planar graph $G$ and symmetric body $C \subset \mathbb{R}^2$,
we define $\radii_C (G) \subset \mathbb{R}^V_{>0}$ be the set of radii $r$ where $S_{G,C}(r) \neq \emptyset$.

\begin{lemma}\label{l:isoimpgen}
	Let $C$ be a regular symmetric convex body in the plane and $P=(G,p,r)$ a $C$-packing.
	If $P$ independent then there exists an open neighbourhood $U \subset S_{G,C}$ of $(p,r)$ and open neighbourhood $U'$ of $r$ such that the restriction of the projection map $\rho : S_{G,C} \rightarrow \mathbb{R}_{>0}^{|V|}$ to $U$ and $U'$ is an open map.
	In particular,
	if there exists an independent $C$-packing $P$ with contact graph $G$ then $\radii_C(G)$ has non-empty interior.
\end{lemma}

\begin{proof}
	By Lemma \ref{lem:2.11},
	$\dim S_{G,C} = 3|V|-|E|$.
	As $\rank R_C(G,p) =|E|$ then by Lemma \ref{l:ranknull},
	$\rank d \rho(p,r) =|V|$ and $d \rho(p,r)$ is surjective.
	The result now holds by the local onto theorem (see \cite[Theorem 3.5.2]{manifold}).
\end{proof}

\begin{proposition}\label{p:radii}
	Let $C$ be a regular symmetric body in $\mathbb{R}^2$,
	$G$ be a planar graph and $k = 2|V|-|E|-1$.
	If the norm $\|\cdot\|_C$ is $C^k$-differentiable on $\mathbb{R}^2 \setminus \{0\}$,
	then $\radii_C(G)$ has positive Lebesgue measure if and only if there exists an independent $C$-packing with contact graph $G$.
\end{proposition}

\begin{proof}
	If there exists an independent $C$-packing with contact graph $G$ then by Lemma \ref{l:isoimpgen},
	$\radii_C(G)$ has positive Lebesgue measure.
	Now suppose $\radii_C(G)$ has positive Lebesgue measure.
	By Lemma \ref{lem:2.12},
	$\Crit_G$ has Lebesgue measure zero,
	hence $\radii_C(G) \setminus \Crit_G$ is non-empty.
	It follows that if we choose $r \in \radii_C(G) \setminus \Crit_G$ then there exists an independent $C$-packing with contact graph $G$ and radii $r$.	
\end{proof}

\begin{corollary}
	Let $C$ be a regular symmetric body in $\mathbb{R}^2$ and $G$ be a $(2,2)$-tight planar graph.
	Then $\radii_C(G)$ has positive Lebesgue measure if and only if there exists a minimally rigid $C$-packing with contact graph $G$.
\end{corollary}

\begin{proof}
	By Proposition \ref{p:support},
	$\|\cdot\|_C$ is $C^1$-differentiable.
	The result now follows from Proposition \ref{p:radii} with $k=1$.
\end{proof}

\section{Topological spaces of convex bodies}\label{sec:regsymbodies}

\subsection{The space of centrally symmetric bodies}

Denote by $\mathcal{K}_d$ the set of centrally symmetric convex bodies in $\mathbb{R}^d$.
The pair $(\mathcal{K}_d, d_H)$ forms a metric space,
where $d_H$ is the Hausdorff metric.
Many operations that can be applied to convex sets can be seen to be continuous in the generated topology.
We refer the reader to the \cite[Theorem 7.2]{wijsman} for the following result.

\begin{lemma}\label{l:polar}
	The map $A \mapsto A^*$ is a homeomorphism from $(\mathcal{K}_d, d_H)$ to itself.
\end{lemma}
%

We can alternatively define the topology of $\mathcal{K}_d$ by another metric.

\begin{proposition}\label{prop:hausvsdash}
	Define $\rho$ to be the metric of $\mathcal{K}_d$ given by
	\begin{align*}
		\rho(C,D) := \sup_{x \in \mathbb{B}^d} \left| \| x\|_{C} - \|x\|_D \right|.
	\end{align*}
	Then $(\mathcal{K}_d,\rho)$ and $(\mathcal{K}_d,d_H)$ are homeomorphic and $\rho(C,D) = d_H(C^*,D^*)$.
\end{proposition}

\begin{proof}
	Choose any $C,D \in \mathcal{K}_d$.
	By \cite[Proposition 2.4.5]{minkowski},
	$d_H(C^*,D^*) = \rho(C,D)$,
	hence the result follows from Lemma \ref{l:polar}.
\end{proof}

For our next result we remember that a \emph{Baire space} is a topological space where every countable intersection of open dense subsets is dense. 

\begin{proposition}\label{p:kbaire}
	The space $\mathcal{K}_d$ is a locally compact Baire space.
\end{proposition}

\begin{proof}
	Define $\mathcal{C}_d$ to be the set of non-empty compact convex sets in $\mathbb{R}^d$ and $\mathcal{C}^s_d \subset \mathcal{C}_d$ to be the subset of centrally symmetric compact convex sets;
	importantly, both $C_d$ and $C_d^s$ contain sets with empty interior.
	By \cite[Theorems 1.8.2, 1.8.3]{schneider},
	$(\mathcal{C}_d,d_H)$ is a locally compact complete metric space.
	Since $\mathcal{C}_d^s$ is a closed subset of $\mathcal{C}_d$ then $(\mathcal{C}^s_d,d_H)$ is a locally compact complete metric space also.
	As $\mathcal{K}_d$ is an open dense subset of $\mathcal{C}^s_d$ then $\mathcal{K}_d$ is a locally compact Baire space by the Baire category theorem.
\end{proof}

\subsection{The space of regular symmetric convex bodies}

We now denote the set of all regular symmetric bodies in $\mathbb{R}^d$ by $\mathcal{B}_d$.
The next result allows us to say that a ``generic'' centrally symmetric body will be a regular symmetric body;
importantly, this will allow us to prove Theorem \ref{t:symperfect} by focusing solely on the regular symmetric bodies.
It was originally shown by Klee in \cite{klee59} that the set of convex bodies that are both smooth and strictly convex is a comeagre subset of the set of convex bodies.
Since much of the language used by Klee is different to our own,
we include a proof of this special case.

\begin{proposition}\label{p:baire}
	The space $\mathcal{B}_d$ is a comeagre Baire subspace of $\mathcal{K}_d$.
\end{proposition}

\begin{proof}
	By Proposition \ref{p:kbaire},
	it suffices to prove $\mathcal{B}_d$ is a comeager subset of $\mathcal{K}_d$.
	To do so we shall prove that both the set $A$ of non-smooth centrally symmetric convex bodies and the set $B$ of not strictly convex centrally symmetric convex bodies are meager sets,
	as $\mathcal{B}_d = \mathcal{K}_d\setminus (A \cup B)$.
	
	Define for each $i \in \mathbb{N}$ the set $A_i \subset A$ of centrally symmetric convex bodies $C$ where the following holds;
	there exists $x\in \partial C$ with distinct supports $y,z \in C^*$,
	where $\|y-z\| \geq \frac{1}{i}$.
	Let $(C_n)_{n \in \mathbb{N}}$ be a convergent sequence in $A_i$ with limit $C \in \mathcal{K}_d$.
	Then for each $C_n$ there exists $x_n \in C_n$ with distinct supports $y_n,z_n \in C_n^*$ where $\|y_n-z_n\| \geq \frac{1}{i}$.
	Since both $(C_n)_{n \in \mathbb{N}}$ and $(C_n^*)_{n \in \mathbb{N}}$ are convergent sequences of compact bodies (Lemma \ref{l:polar}),
	we may take a subsequence $(C_{n_k})_{k \in \mathbb{N}}$ such that $x_{n_k} \rightarrow x$, $y_{n_k} \rightarrow y$ and $z_{n_k} \rightarrow z$ as $k \rightarrow \infty$ for some $x \in \partial C$, $y,z \in \partial C^*$.
	By continuity,
	both $y,z$ are supports of $x$ (with respect to $C$) and $\|y-z\| \geq \frac{1}{i}$,
	hence $C \in A_i$ and $A_i$ is closed.
	By \cite[Theorem 2.5.2]{minkowski},
	$\mathcal{B}_d$ is a dense subset of $\mathcal{K}_d \setminus A_i$,
	hence $A_i$ is a closed nowhere dense set for every $i \in \mathbb{N}$.
	As $A = \bigcup_{i=1}^n A_i$ then $A$ is meager.
	
	Define for each $i \in \mathbb{N}$ the set $B_i \subset A$ of centrally symmetric convex bodies $C$ where the following holds;
	there exists $x,y\in \partial C$ that both have support $z \in \partial C^*$ and $\|x-y\| \geq \frac{1}{i}$.
	Let $(C_n)_{n \in \mathbb{N}}$ be a convergent sequence in $B_i$ with limit $C \in \mathcal{K}_d$.
	Then for each $C_n$ there exists $x_n,y_n \in C_n$ with shared support $z_n \in \partial C_n^*$ so that $\|x_n-y_n\| \geq \frac{1}{i}$.
	Similarly to before,
	we may take a subsequence $(C_{n_k})_{k \in \mathbb{N}}$ such that $x_{n_k} \rightarrow x$, $y_{n_k} \rightarrow y$ and $z_{n_k} \rightarrow z$ as $k \rightarrow \infty$ for some $x,y \in \partial C$, $z \in \partial C^*$;
	likewise, $z$ is a support of both $x$ and $y$ (with respect to $C$) and $\|x-y\| \geq \frac{1}{i}$,
	hence $C \in B_i$ and $B_i$ is closed.
	By \cite[Theorem 2.5.2]{minkowski},
	$\mathcal{B}_d$ is a dense subset of $\mathcal{K}_d \setminus B_i$,
	hence $B_i$ is a closed nowhere dense set for every $i \in \mathbb{N}$.
	As $B = \bigcup_{i=1}^n B_i$ then $B$ is meager.
\end{proof}

Whilst dealing with convergent sequences in $\mathcal{K}_d$,
it is often desired that their respective duality maps also converge.
We shall prove with Proposition \ref{p:phicont} that this is indeed true for convergent sequences in $\mathcal{B}_d$.
We first state the following useful result regarding convex functions.

\begin{theorem}\cite[Theorem 25.7]{rockafellar}\label{t:convex}
	Let $U \subset \mathbb{R}^d$ be an open convex set and $(f_n)_{n \in \mathbb{N}}$ be a sequence of convex functions on $U$ with pointwise limit $f :U \rightarrow \mathbb{R}$.
	Suppose $f$ is differentiable and each $f_n$ is differentiable also.
	Then for any closed and bounded subset $K$ of $U$,
	\begin{align*}
		\sup_{x \in K} \|f'(x) - f_n'(x)\| \rightarrow 0
	\end{align*}
	as $n \rightarrow \infty$.
\end{theorem}

\begin{proposition}\label{p:phicont}
	Let $(C_n)_{n \in \mathbb{N}}$ be a convergent sequence in $\mathcal{B}_d$ with limit $C \in \mathcal{B}_d$.
	Then
	\begin{align*}
		\sup_{x \in \mathbb{B}^d} \|\varphi_{C_n}(x) - \varphi_C(x)\| \rightarrow 0
	\end{align*}
	as $n \rightarrow \infty$.
\end{proposition}

\begin{proof}
	By Proposition \ref{prop:hausvsdash},
	$\|x\|_{C_n} \rightarrow \|x\|_C$ as $n \rightarrow \infty$ for any $x \in \mathbb{R}^d$.
	The result now holds by Theorem \ref{t:convex},
	with $U = \mathbb{R}^d$ and $K=\mathbb{B}^d$.
\end{proof}

\subsection{Special classes of regular symmetric bodies in the plane}\label{sec:specialclass}

Although regular symmetric bodies are very useful,
we will occasionally require that our regular symmetric bodies have stronger properties.
There are two important subspaces of $\mathcal{B}_2$ we shall require later in Sections \ref{sec:genedge} and \ref{sec:kldense};
the regular symmetric bodies with positive curvature and the regular symmetric bodies with analytic boundary.

Let $C \in \mathcal{K}_2$ have a $C^k$-differentiable norm on $\mathbb{R}^2 \setminus \{0\}$ for $k \geq 2$.
By the constant rank theorem (see \cite[Theorem 2.5.15]{manifold}),
the boundary $\partial C$ of $C$ is $C^k$-differentiable,
i.e.~it is a $C^k$-differentiable manifold.
The boundary of $\partial C$ has positive curvature if and only if the following holds;
for any $C^2$-diffeomorphism $\alpha : \mathbb{T} \rightarrow \partial C$ (where $\mathbb{T} := (-\pi,\pi]$ has the topology generated by the metric $|a - b|_{\sim} := \min \{ |a-b|, 2\pi - |a - b| \}$) we have that $\alpha'(t), \alpha''(t)$ are linearly independent.
As $\alpha''(t) \neq 0$ for any $t \in \mathbb{T}$,
it follows that $C$ will be strictly convex.

\begin{definition}
	We denote by $\mathcal{B}_2^+$ the set of regular symmetric bodies in the plane that have a $C^2$-differentiable boundary with positive curvature.
\end{definition}

\begin{lemma}\label{l:schneider}
	For any $C \in K_2$,
	the following are equivalent:
	\begin{enumerate}[(i)]
		\item \label{p:schneider1} $C \in \mathcal{B}_2^+$.
		\item \label{p:schneider2} There exists a $C^2$-diffeomorphism $\alpha : \mathbb{T} \rightarrow \partial C$,
		and every $C^2$-diffeomorphism $ \mathbb{T} \rightarrow \partial C$ has positive curvature.
		\item \label{p:schneider3} There exists a $C^2$-diffeomorphism $\alpha : \mathbb{T} \rightarrow \partial C$ with positive curvature.
		\item \label{p:schneider4} $C^* \in \mathcal{B}_2^+$.
	\end{enumerate}
\end{lemma}

\begin{proof}
	It is immediate that (\ref{p:schneider1}) $\Rightarrow$ (\ref{p:schneider2}) $\Rightarrow$ (\ref{p:schneider3}).
	As reparametrising a curve preserves its curvature, (\ref{p:schneider3}) $\Rightarrow$ (\ref{p:schneider2}).
	
	Suppose (\ref{p:schneider2}) holds.
	As shown in \cite[Section 5, pg 106]{schneider},
	the norm $\| \cdot \|_C^*$ is $C^2$-differentiable on $\mathbb{R}^2 \setminus \{0\}$.
	By Proposition \ref{p:polar},
	$\| \cdot \|_{C}^*=\| \cdot \|_{C^*}$.
	As shown in \cite[Section 5 pg 111]{schneider},
	this implies $C^*$ has a $C^2$-differentiable boundary with positive curvature,
	hence (\ref{p:schneider4}) holds.
	
	By our previous work we have that (\ref{p:schneider1}) $\Rightarrow$ (\ref{p:schneider4}).
	By symmetry we thus have (\ref{p:schneider4}) $\Rightarrow$ (\ref{p:schneider1}) also,
	which completes the proof.
\end{proof}

For any spanning set $a:=\{a_1,\ldots,a_n\} \subset \mathbb{R}^2$ with $n \geq 3$ and any $w > \log 2n$ (where $\log$ has base $e$),
define the analytic function
\begin{align*}
	\phi_{a,w}: \mathbb{R}^2 \rightarrow \mathbb{R}, ~x \mapsto \frac{1}{2n} \sum_{i=1}^n e^{w(a_i.x-1)} + e^{w(-a_i.x-1)}.
\end{align*}
We compute that $\phi_{a,w}(x) > 0$, $\phi_{a,w}(-x)=\phi_{a,w}(x)$ and 
\begin{align*}
	d \phi_{a,w}(x) := \frac{w}{2n} \sum_{i=1}^n a_i \left( e^{w(a_i.x-1)} - e^{w(-a_i.x-1)} \right)
\end{align*}
for all $x \in \mathbb{R}^2$.
We list another useful property of $\phi_{a,w}$ below.

\begin{lemma}\label{l:phi}
	Let $a=\{a_1, \ldots, a_n \} \subset \mathbb{R}^2$ be a spanning set with $n \geq 3$ and let $w > \log 2n$.
	Then for all $x \in \mathbb{R}^2 \setminus \{0\}$, the function $f: [0,\infty) \rightarrow [\phi_{a,w}(0),\infty), t \mapsto \phi_{a,w}(tx)$ is an analytic diffeomorphism.
\end{lemma}

\begin{proof}
	We compute that
	\begin{align*}
		f''(t) =  \frac{1}{2n} \sum_{i=1}^n (w a_i.x)^2 \left( e^{w(a_i.tx-1)} + e^{w(-a_i.tx-1)} \right).
	\end{align*}
	As $\{a_1, \ldots, a_n \} $ is a spanning subset of $\mathbb{R}^2$ then $f''(t) >0$ for all $t \geq 0$.
	The required result now holds as $f$ is analytic.
\end{proof}

The importance of the $\phi_{a,w}$ functions is that they allow us to define centrally symmetric convex bodies.

\begin{lemma}\label{l:eposcurv}
	Let $a=\{a_1, \ldots, a_n \} \subset \mathbb{R}^2$ be a spanning set with $n \geq 3$ and let $w > \log 2n$.
	If we define the set $C := \{ x \in \mathbb{R}^2 : \phi_{a,w}(x) \leq 1\}$ then the following holds:
	\begin{enumerate}[(i)]
		\item \label{l:eposcurvitem1} $\partial C = \{ x \in \mathbb{R}^2 : \phi_{a,w}(x) = 1\}$.
		\item	\label{l:eposcurvitem2} $C \in \mathcal{B}_2^+$.
		\item \label{l:eposcurvitem3} $\varphi_C(x) = \frac{d\phi_{a,w}(x)}{d\phi_{a,w}(x).x}$ for all $x \in \partial C$.
	\end{enumerate}
\end{lemma}

\begin{proof}
	(\ref{l:eposcurvitem1}):
	Since $\phi_{a,w}$ is continuous it is immediate that $\partial C \subset \{ x \in \mathbb{R}^2 : \phi_{a,w}(x) = 1\}$.
	Suppose there exists $x \in C^o$ with $\phi_{a,w}(x)=1$.
	Then there exists $\epsilon >0$ such that $\phi_{a,w}(tx) \leq 1$ for all $1 \leq t \leq 1 +\epsilon$.
	However by Lemma \ref{l:phi},
	$\phi_{a,w}(tx) > \phi_{a,w}(x)=1$ for all $t>1$,
	a contradiction.
	Hence, $\partial C = \{ x \in \mathbb{R}^2 : \phi_{a,w}(x) = 1\}$.
	
	(\ref{l:eposcurvitem2}):
	By applying the constant rank theorem for analytic maps (see \cite[Chapter 1.8]{KrantzPark92}) and (\ref{l:eposcurvitem1}),
	it follows that $\partial C$ is an analytic 1-dimensional compact manifold;
	hence,
	there exists an analytic diffeomorphism $\mathbb{T} \rightarrow \partial C$.
	Choose any $C^2$-diffeomorphism $\alpha : \mathbb{T} \rightarrow \partial C$.
	By reparametrising we may assume $\alpha$ is regular (i.e.~$\|\alpha'(t)\|=1$ for all $t \in \mathbb{T}$),
	and thus $C$ has positive curvature if and only if $\|\alpha''(t)\| >0$ for all $t \in \mathbb{T}$.
	By differentiating the map $t \mapsto (\phi_{a,w} \circ \alpha)(t)$ twice we see that $(\phi_{a,w} \circ \alpha)''(t)$ is equal to
	\begin{align*}
		 \frac{w}{2n} \sum_{i=1}^n w(a_i . \alpha'(t))^2 \left( e^{w(a_i.x-1)} + e^{w(-a_i.x-1)} \right) + a_i.\alpha''(t) \left( e^{w(a_i.x-1)} - e^{w(-a_i.x-1)} \right).
	\end{align*}
	For each $t \in \mathbb{T}$ we have $\phi_{a,w}\circ \alpha(t)=1$ as $\alpha(t) \in \partial C$,
	hence $(\phi_{a,w} \circ \alpha)''(t)=0$.
	Since
	\begin{align*}
		w(a_i . \alpha'(t))^2 \left( e^{w(a_i.x-1)} + e^{w(-a_i.x-1)} \right) >0
	\end{align*}
	for each $i \in \{1, \ldots, n\}$ and $(\phi_{a,w} \circ \alpha)''(t)=0$,
	we must have
	\begin{align*}
		\alpha''(t). \left(\sum_{i=1}^n a_i \left( e^{w(a_i.x-1)} - e^{w(-a_i.x-1)} \right) \right) < 0
	\end{align*}
	for all $t \in \mathbb{T}$.
	It follows that $\|\alpha''(t)\|>0$,
	thus $C \in \mathcal{B}_2^+$ by Lemma \ref{l:schneider}.
	
	(\ref{l:eposcurvitem3}):
	Choose any point $x \in \partial C$.
	Then the tangent space of $\partial C$ at $x$ is $\ker d\phi_{a,w}(x)$,
	hence $d \phi_{a,w}(x) = \lambda \varphi_C(x)$ for some $\lambda \neq 0$.
	Since $\varphi_C(x).x = 1$ then $\lambda = d \phi_{a,w}(x).x$ as required.
\end{proof}

This motivates the following class of convex bodies.

\begin{definition}\label{def:exp}
	For any integer $j \geq 3$,
	we define $\mathcal{B}_2^\phi(j)$ to be the set of centrally symmetric convex bodies $C \subset \mathbb{R}^2$ where there exists a spanning set $a=\{a_1, \ldots, a_j \} \subset \mathbb{R}^2$ and $w > \log 2j$ so that $C = \{ x \in \mathbb{R}^2 : \phi_{a,w}(x) \leq 1\}$.
	We define $\mathcal{B}_2^\phi := \bigcup_{j=3}^\infty \mathcal{B}_2^\phi(j)$.
\end{definition}

We shall be required later to approximate centrally symmetric bodies by elements $\mathcal{B}_2^\phi$.
To show this we first prove the following.

\begin{lemma}\label{l:poly}
	For each $n \geq 2$,
	define the subset $\mathcal{P}_2(n) \subset \mathcal{K}_2$ of convex bodies $C$ where for some subset $S \subset \mathbb{R}^2$ of $n$ pairwise linearly independent points we have $C=\{ x\in \mathbb{R}^2: |x.y| \leq 1 \text{ for all } y \in S \}$.
	Then for any open set $U \subseteq \mathcal{K}_2$,
	there exists $N \geq 2$ such that $\mathcal{P}_2(n) \cap U \neq \emptyset$ for all $n \geq N$.
\end{lemma}

\begin{proof}
	Choose any $C \in U$ and fix $\epsilon >0$ such that if $d_H(C',C)< \epsilon$ then $C' \in U$.
	As $\partial C$ is compact then there exists a finite set $S \subset \partial C$ where $x\in S$ if and only if $-x \in S$,
	so that $\partial C \subset S+ \epsilon \mathbb{D}$.
	We now note that if we define $C'$ to be the convex hull of $S$ then $C' \in \mathcal{K}_2$ and $C' \subset C \subset C' + \epsilon \mathbb{D}$,
	hence $d_H(C,C') \leq \epsilon$.	
	As $C'$ is a centrally symmetric polygon,
	there exists a set $S \subset \mathbb{R}^2$ of $N$ pairwise linearly independent points so that $C' = \{ x\in \mathbb{R}^2: |x.y| \leq 1 , y \in S\}$.
	We note that if we choose points $y_{1}, \ldots,y_{n-N} \in \partial C^*$ so that $S' := S \cup \{y_{1}, \ldots,y_{n-N}\}$ is a set of $n$ pairwise linearly independent points,
	then $C' = \{ x\in \mathbb{R}^2: |x.y| \leq 1 , y \in S\}$ also.
	Hence $\mathcal{P}_2(n) \cap U \neq \emptyset$ for all $n \geq N$ as required.
\end{proof}

%
%
%
%

\begin{proposition}\label{p:phisetdist}
	For any open set $U \subseteq \mathcal{K}_2$,
	there exists $N \geq 3$ such that $\mathcal{B}_2^\phi(j) \cap U \neq \emptyset$ for all $j \geq N$.
	Hence, the set $\mathcal{B}_2^\phi$ is dense in $\mathcal{K}_2$.
\end{proposition}

\begin{proof}
	Choose any $C \in U$ and fix $\epsilon >0$ such that if $d_H(C',C)< \epsilon$ then $C' \in U$.
	By Lemma \ref{l:poly},
	there exists $N \geq 3$ such that for each $n \geq N$ there exists $C' \in \mathcal{P}_d(n)$ with $d_H(C,C')<\epsilon/2$.
	Choose any $j \geq N$ and define $a_1,\ldots,a_j \in \mathbb{R}^2$ to be pairwise linearly independent points where $C' = \{ x\in \mathbb{R}^2: |x.a_i| \leq 1, i \in \{1,\ldots,j\} \}$.
	For $w > \log j$,
	define $C(w) := \{ x \in \mathbb{R}^2 : \phi_{a,w}(x) \leq 1\}$.
	We note that $\phi_{a,w}(x) >1$ for all $x \in \partial C'$,
	hence $C(w) \subset C'$ for all $w > \log n$.
	As $a$ is a set of distinct elements,
	then $\phi_{a,w} \rightarrow 1_{\partial C'}$ uniformly on $C'$ as $w \rightarrow \infty$,
	where $1_{\partial C'}: C' \rightarrow \{0,1\}$ is the indicator function of $\partial C'$ .
	It follows that $d_H(C(w),C') \rightarrow 0$ as $w \rightarrow \infty$,
	so we may choose $w' > \log 2j$ such that $d_H(C(w'),C') <\epsilon/2$.
	By the triangle inequality we have $d_H(C(w'),C) < \epsilon$ as required.
\end{proof}

Our interest in each set $\mathcal{B}_2^\phi(j)$ is that we can, in some sense, characterize them as open connected subsets of finite dimensional spaces.

\begin{lemma}\label{l:eparam}
	For each $j \geq 3$ define the set
	\begin{align*}
		X_j := \left\{(a,w) \in \mathbb{R}^{2n} \times (\log 2j, \infty) : a_1,\ldots,a_j \text{\emph{ span }} \mathbb{R}^2 \right\} 
	\end{align*}
	Then $X_j$ is an open connected set and
	\begin{align*}
		\lambda_j: X_j \rightarrow \mathcal{B}_2^\phi(j), ~ (a, w) \mapsto \left\{ x \in \mathbb{R}^2: \phi_{a,w}(x) \leq 1 \right\}
	\end{align*}
	is a continuous surjective map.
\end{lemma}

\begin{proof}
	By definition,
	$f$ is surjective,
	and it is immediate that $X_j$ is an open connected set.
	Suppose $((a^n,w_n))_{n \in \mathbb{N}}$ is a sequence in $X_j$ that converges to $(a,w) \in X_j$.
	Choose any bounded open neighbourhood $U$ of $\lambda(a,w)$.
	Since $\phi_{a^n,w_n} \rightarrow \phi_{a,w}$ uniformly on $U$ as $n \rightarrow \infty$,
	then $d_H(\lambda(a^n,w_n), \lambda(a,w)) \rightarrow 0$ as $n \rightarrow \infty$.
\end{proof}

\section{Packings with maximal planar contact graphs}\label{sec:maxplanar}

\subsection{Mapping triangulations to their unique packings}

It is well-known that a planar graph $G=(V,E)$ is maximal if and only if one of the following equivalent conditions hold:
\begin{enumerate}[(i)]
	\item $G$ is not a proper subgraph of any other planar graph.
	\item Every face of $G$ is a triangle.
	\item $|E|=3|V|-6$.
\end{enumerate}
If a disc packing has a maximal planar contact graph, then it is unique up to M\"{o}bius transformations and reflections by the KAT theorem.
In \cite{schramm91},
Schramm proved a similar uniqueness result for homothetic convex body packings where both smoothness and strict convexity are assumed.

\begin{theorem}\cite{schramm91}\label{t:odedmax}
	Let $G = (V,E)$ be a maximal planar graph and $C$ be a convex body in the plane that is both smooth and strictly convex.
	Choose any clique $\{a,b,c\} \subset V$ and three corresponding points $x_a,x_b,x_c \in \mathbb{R}^2$ that are not colinear.
	Then there exists a unique $C$-packing $P=(G,p,r)$ with $p_a=x_a$, $p_b = x_b$, $p_c = x_c$, and $\{p_v : v \in V\}$ lies inside the interior of the triangle formed by $x_a,x_b,x_c$.
\end{theorem}

This is stronger than Theorem \ref{t:oded} due to the uniqueness of the obtained packing,
the trade off being that strictly convexity is now required for the proof.
Using Theorem \ref{t:odedmax} and the added assumption of central symmetry,
we can obtain the following result that lets us determine how these unique packings change as the regular symmetric body is altered.

\begin{proposition}\label{p:contmax}
	Let $G = (V,E)$ be a maximal planar graph.
	Choose any clique $\{a,b,c\} \subset V$ and three corresponding points $x_a,x_b,x_c \in \mathbb{R}^2$ that are not colinear.
	Then there exists a unique continuous map $f: \mathcal{B}_2 \rightarrow \mathbb{R}^{2|V|} \times \mathbb{R}^{|V|}_{>0}$ such that,
	given $f(C) = (p,r)$,
	the triple $(G,p,r)$ is the unique $C$-packing with $p_a=x_a$, $p_b = x_b$, $p_c = x_c$, and $\{p_v : v \in V\}$ contained in the interior of the triangle formed by $x_a,x_b,x_c$.
\end{proposition}

\begin{proof}
	The existence and uniqueness of the map $f$ follows immediately from Theorem \ref{t:odedmax}, so we shall only be required to prove $f$ is continuous.
	Suppose otherwise;
	i.e.~there exists $\epsilon >0$ and a sequence $(C_n)_{n \in \mathbb{N}}$ in $\mathcal{B}_2$ with limit $C$ where,
	given $(p^n,r^n) := f(C_n)$ and $(p,r) := f(C)$,
	we have 
	\begin{align*}
		\|f(C_n)-f(C)\|_{\max} := \max \left\{ \left\| p^n_v-p_v \right\|,\left|r^n_v-r_v \right| :v \in V \right\} \geq \epsilon.
	\end{align*}
	
	For each $n \in \mathbb{N}$, the points $\{p_v^n : v \in V\}$ lie inside the triangle formed by $x_a,x_b,x_c$,
	hence $(p^n)_{n \in \mathbb{N}}$ is a bounded sequence.
	If we define
	\begin{align*}
		R:= \sup \left\{ \|x_a - x_b\|_{C_n}, \|x_a - x_c\|_{C_n} , \|x_b-x_c\|_{C_n}: n \in \mathbb{N} \right\}
	\end{align*}
	then $R<\infty$ by Proposition \ref{prop:hausvsdash},
	and for every point $x$ in the interior of the triangle formed by $x_a,x_b,x_c$ we have $x_a,x_b,x_c \in RC_n + x$ for all $n \in \mathbb{N}$.
	Hence each set $\{r_v^n : v \in V\}$ is bounded by $R>0$,
	and $(r^n)_{n \in \mathbb{N}}$ is a bounded sequence.
	
	Since the sequence $((p^n,r^n))_{n \in \mathbb{N}}$ is bounded,
	there exists a subsequence $((p^{n_k},r^{n_k}))_{k \in \mathbb{N}}$ with limit $(p',r')$.
	It follows from Proposition \ref{prop:hausvsdash} that as $k \rightarrow \infty$ we have
	\begin{align*}
		0 = h_{G,C_{n_k}}(p^{n_k},r^{n_k}) \rightarrow h_{G,C}(p',r');
	\end{align*}
	hence $(G,p',r')$ is a $C$-packing with $p'_a=x_a$, $p'_b = x_b$ and $p'_c = x_c$, and $\{p'_v : v \in V\}$ lies inside the triangle formed by $x_a,x_b,x_c$.
	By Theorem \ref{t:odedmax},
	$(p',r') = (p,r)$,
	contradicting our original assumption.
\end{proof}

\subsection{An important corollary to Proposition \ref{p:contmax}}\label{sec:genedge}

We begin with the following definition.

\begin{definition}
	We say a framework $(G,p)$ in $\mathbb{R}^d$ has the \emph{general edge condition} if $p_{v_1}-p_{w_1}$ and $p_{v_2}-p_{w_2}$ are linearly independent for any pair of edges $v_1w_1,v_2 w_2 \in E$.
	For any planar graph $G$,
	we define $C \in \mathcal{K}_2$ to have the \emph{general edge condition for $G$} there exists a $C$-packing with contact graph $G$ and contact framework with the general edge condition.
\end{definition}

In this section we shall prove the following corollary to Proposition \ref{p:contmax}.

\begin{corollary}\label{cor:generalpos}
	Let $G = (V,E)$ be a planar graph.
	Then there exists an open dense subset of $\mathcal{B}_2$ of regular symmetric bodies $C$ with the general edge condition for $G$.
\end{corollary}

Corollary \ref{cor:generalpos} we be utilised later in Section \ref{sec:kldense} for an important technical result (see Lemma \ref{kl:dense}).
To prove Corollary \ref{cor:generalpos} we will require the following lemmas.

\begin{lemma}\label{l:mobius}
	Let $\hat{\mathbb{C}} := \mathbb{C} \cup \{\infty \}$ be the Riemann sphere with M\"{o}bious transforms $PGL(2, \mathbb{C})$.
	Then for any four points distinct points $x,y,u,w \in \mathbb{C}$,
	the following holds:
	\begin{enumerate}[(i)]
		\item \label{l:mobius1} There exists an open dense subset $U \subset PGL(2, \mathbb{C})$ where $\phi(x),\phi(y),\phi(u)$ are not equal to $\infty$ and the pair $\phi(x)-\phi(y),\phi(x)-\phi(u)$ are linearly independent over $\mathbb{R}$.
		\item \label{l:mobius2} There exists an open dense subset $U \subset PGL(2, \mathbb{C})$ where $\phi(x),\phi(y),\phi(u),\phi(w)$ are not equal to $\infty$ and the pair $\phi(x)-\phi(y),\phi(u)-\phi(w)$ are linearly independent over $\mathbb{R}$.
	\end{enumerate}
\end{lemma}

\begin{proof}
	It suffices to prove for (\ref{l:mobius1}) and (\ref{l:mobius2}) that a single such M\"{o}bius transform exists.
	For any two triples $\{a_1,a_2,a_3\}$, $\{b_1,b_2,b_3\}$,
	there exists a M\"{o}bius transform that maps $a_i \mapsto b_i$,
	hence it is immediate that (\ref{l:mobius1}) holds.
	Further,
	we may assume $x=0$, $y=1$ and $w=i$.
	Suppose $x-y,w-z$ are linearly dependent when considered as real vectors.
	Then $z=a+i$ for some non-zero $a \in \mathbb{R}$.
	If we define the M\"{o}bius transform $\phi(z) = \frac{1}{z-a}$ then $\phi(0)-\phi(1) = \frac{1}{a(a-1)}$ and $\phi(i)-\phi(a+i) = \frac{a}{a^2+1}( -1+ai)$.
	Since $a\neq 0$ then $\phi(x)-\phi(y),\phi(u)-\phi(w)$ are linearly independent when considered as real vectors and (\ref{l:mobius2}) holds.
\end{proof}

\begin{lemma}\label{l:circlegenedge}
	Let $G = (V,E)$ be a maximal planar graph.
	Then there exists a disc packing $P=(G,p,r)$ with the general edge condition.
\end{lemma}

\begin{proof}
	We may consider any disc packing as being a circle packing on the Riemann sphere $\hat{\mathbb{C}} := \mathbb{C} \cup \{\infty \}$ by considering a point $(x,y)$ as $x+iy$ and any disc by its boundary;
	we do the latter as some M\"{o}bius transforms can map discs to every point outside of a disc's interior.
	It follows from Lemma \ref{l:mobius} that there exists an open dense set $U \subset PGL(2, \mathbb{C})$ of M\"{o}bius transforms $\phi$ where the set $\{ \phi(p_v)-\phi(p_w) : vw \in E\}$ is contained in $\mathbb{C}$ and has $|E|$ pairwise linearly independent elements over $\mathbb{R}$.
	It follows that $\phi(P)$ defines a disc packing in $\mathbb{R}^2$ with the general edge condition,
	as a circle is only mapped to a line if its center is mapped to $\infty$.
\end{proof}

\begin{lemma}\label{l:maxinvert}
	Let $C$ be a regular symmetric body, $G=(V,E)$ be a maximal planar graph and clique $\{a,b,c\} \subset V$.
	If $P=(G,p,r)$ is a $C$-packing and $\tilde{R}_C(G,p,r)$ is the matrix formed from $R_C(G,p,r)$ by deleting the point columns that correspond to $a,b,c$,
	then $\tilde{R}_C(G,p,r)$ is invertible.
\end{lemma}

\begin{proof}
	Suppose there exists a non-zero element $(u,s) \in \ker R_C(G,p,r)$ with $u_a=u_b=u_c=0$.
	By Lemma \ref{lem:2.11},
	there exists $\epsilon >0$ and an injective continuous path $\alpha:(-\epsilon, \epsilon) \rightarrow S_{G,C}$ with $\alpha(t)=(p(t),r(t))$ where $p_a(t)=p_a$, $p_b(t)=p_b$ and $p_c(t)=p_c$.
	However this contradicts Theorem \ref{t:odedmax},
	hence if $(u,s) \in \ker R_C(G,p,r)$ with $u_a=u_b=u_c=0$ then $(u,s)=0$.
	
	Choose any $(\tilde{u},s) \in \ker \tilde{R}_C(G,p,r)$.
	We can extend $(\tilde{u},s)$ to a vector $(u,s)$ by setting $u_a=u_b=u_c=0$ and $u_v= \tilde{u}_v$ for all $v \in V \setminus \{a,b,c\}$.
	We note that $R_C(G,p,r)(u,s) = \tilde{R}_C(G,p,r)(\tilde{u},s)=0$,
	hence $(\tilde{u},s)=0$.
	Since $|E|=3|V|-6$ and $\tilde{R}_C(G,p,r)$ is injective,
	$\tilde{R}_C(G,p,r)$ is invertible.
\end{proof}

The following is a much stronger result than Proposition \ref{p:contmax} for a specific class of convex body.
We direct the reader to Section \ref{sec:specialclass} for definitions of the sets $\mathcal{B}_2^\phi$, $\mathcal{B}_2^\phi(j)$, $X_j$ and the $\phi_{a,w}$ functions.

\begin{lemma}\label{l:expgenpos}
	Let $G = (V,E)$ be a maximal planar graph and $\{a,b,c\} \subset V$ be a clique,
	and choose any $x_a,x_b,x_c \in \mathbb{R}^2$ that are not colinear.
	For any $j \geq 3$,
	the map $f \circ \lambda_j : X_j \rightarrow  \mathbb{R}^{2|V|} \times \mathbb{R}^{|V|}_{>0}$ is 
	is analytic,
	where $\lambda_j: X_j \rightarrow \mathcal{B}_2^\phi$ is the continuous surjective map defined in Lemma \ref {l:eparam} and $f: \mathcal{B}_2 \rightarrow \mathbb{R}^{2|V|} \times \mathbb{R}^{|V|}_{>0}$ is the continuous map defined in Proposition \ref{p:contmax}.
\end{lemma}

\begin{proof}
	Fix $j \geq 3$.	
	Let 
	\begin{align*}
		S := \left\{ (p,r) \in \mathbb{R}^{2|V|} \times \mathbb{R}_{>0}^{|V|} : p_a=x_a,p_b=x_b,p_c = x_c \right\}
	\end{align*}	
	and for each $vu \in E$ define the analytic function $\mu_{vu} : S \times X_j \rightarrow \mathbb{R}$ by
	\begin{align*}
		\mu_{uv}(p,r,a,w) := \phi_{a,w} \left( \frac{p_v-p_u}{r_v+r_u} \right) - 1.
	\end{align*}
	We now define
	\begin{align*}
		\mu: S \times X_j \rightarrow \mathbb{R}^{|E|}, ~ (p,r,a,w) \mapsto (\mu_{vu}(p,r,a,w))_{vu \in E}
	\end{align*}
	and note that the zero set of $\mu$ contains the open non-empty set 
	\begin{align*}
		O := \left\{ (p,r,a,w) : (G,p,r) \text{ is a $C$-packing for }C := \left\{ x \in \mathbb{R}^2 : \phi_{a,w}(x) \leq 1\right\} \right\}.
	\end{align*}	
	Let $(p,r)$ be a $C$-packing for $C := \{ x \in \mathbb{R}^2 : \phi_{a,w}(x) \leq 1\}$ with $(a,w)\in X_j$.
	For each edge $vu$ we compute the following partial derivatives at $(p,r,a,w)$ (see Lemma \ref{l:eposcurv}(\ref{l:eposcurvitem3}) for the substitution):
	\begin{eqnarray*}
		\frac{\partial \mu_{vu}}{\partial p_v}(p,r,a,w) &=&\frac{1}{r_v+r_u} d \phi_{a,w}\left( \frac{p_v-p_u}{r_v+r_u} \right) \\
		&=& \left( d \phi_{a,w}\left( \frac{p_v-p_u}{r_v+r_u} \right).\frac{p_v-p_u}{(r_v+r_u)^3} \right) \varphi_C(p_v-p_u), \\
		\frac{\partial \mu_{vu}}{\partial r_v}(p,r,a,w) &=& -d \phi_{a,w} \left( \frac{p_v-p_u}{r_v+r_u} \right).\frac{p_v-p_u}{(r_v+r_u)^2} \\
		&=& -\left( d \phi_{a,w}\left( \frac{p_v-p_u}{r_v+r_u} \right).\frac{p_v-p_u}{(r_v+r_u)^3} \right) (r_v+r_u)
	\end{eqnarray*}
	If we define the $|E| \times (3|V|-6)$ matrix $M(p,r,a,w)$ to be partial derivative of $\mu$ at $(p,r,a,w)$ in the $(p,r)$ coordinates,
	then $M(p,r,a,w)$ is the matrix obtained from $R_C(G,p,r)$ by deleting the columns that correspond to the point components of $a,b,c$ and multiplying each row $vu$ of $M$ by the non-zero scalar $d \phi_{a,w} \left( \frac{p_v-p_u}{r_v+r_u} \right).\frac{p_v-p_u}{(r_v+r_u)^3}$.
	Hence by Lemma \ref{l:maxinvert},
	$M(p,r,a,w)$ is invertible.	
	By the implicit function theorem for analytic maps (see \cite[Theorem 1.8.3]{KrantzPark92}),
	there exists an open neighbourhood $U \subset X_j$ of $(a,w)$
	and a unique analytic map $\gamma_{a,w} : U \rightarrow S$ such that $\gamma_{a,w}(a,w)=(p,r)$ and $\mu(\gamma_{a,w}(a',w'),a',w') = 0$ for all $(a',w') \in U$.
	Further,
	since $O \subset S \times X_j$ is an open subset then we may assume that $(\gamma_{a,w}(a',w'),a',w') \in O$ for all $(a',w') \in U$.
	
	By Theorem \ref{t:odedmax},
	for each $(a,w) \in X_j$ there exists $(p,r) \in S$ such that $(G,p,r)$ is a $C$-packing for $C := \{ x \in \mathbb{R}^2 : \phi_{a,w}(x) \leq 1\}$.
	Hence for each $(a,w) \in X_j$ there exists a unique analytic map $\gamma_{a,w}$ as previously described.
	By the uniqueness of each map $\gamma_{a,w}$,
	we can extend to an analytic map $\gamma: X_j \rightarrow S$,
	where $(\gamma(a,w),a,w) \in O$ for all $(a,w) \in X_j$.
	We now note that by the uniqueness of $f$ we must have $\gamma(a,w) = f \circ \lambda_j(a,w)$ for all $(a,w) \in X_j$.
\end{proof}

\begin{lemma}\label{l:genedgesub}
	Let $C \in \mathcal{B}_2$,
	$G=(V,E)$ be a maximal planar graph and $G'=(V,E')$ a connected subgraph of $G$.
	If there exists a $C$-packing $P=(G,p,r)$ with the general edge condition,
	then there exists a $C$-packing $P'=(G',p',r')$ with the general edge condition.
\end{lemma}

\begin{proof}
	Let $a \in \mathbb{R}^{|E|}$ to be the vector with $a_{e}=0$ for all $e \in E'$ and $a_e =1$ for all $e \in E \setminus E'$.
	Define $X \subset \mathbb{R}^{2|V|} \times \mathbb{R}^{|V|}_{>0}$ to be the set of pairs $(q,s)$ where $q_a=p_a$, $q_b=p_b$ and $q_c =p_c$.
	By Lemma \ref{l:maxinvert},
	there exists a neighbourhood $U \subset X$ of $(p,r)$ where for all $(q,s) \in U$,
	the matrix $\tilde{R}_C(G,q,s)$ is invertible.
	Define the continuous map
	\begin{align*}
		h : U \rightarrow \mathbb{R}^6 \times \mathbb{R}^{3|V|-6}, ~ (q,s) \mapsto \left(0,\tilde{R}_C(G,q,s)^{-1}(a) \right).
	\end{align*}
	We note that for sufficiently small neighbourhood $U$,
	$h$ is Lipschitz continuous. 
	By \cite[Lemma 4.1.6]{manifold},
	there exists $\epsilon>0$ and a unique $C^1$-differentiable path $\alpha :[0,\epsilon) \rightarrow U$,
	where $\alpha(0)= (p,r)$ and $\alpha'(t) = h(\alpha(t))$ for all $t \in [0,\epsilon)$.
	For each $t \in [0,\epsilon)$,
	if $(q,s) = \alpha(t)$ then $\|q_v - q_w\|_C = s_v+s_w$ for all $vw \in E'$,
	and $\|q_v - q_w\|_C > s_v+s_w$ for all $vw \in E \setminus E'$;
	further,
	there exists $0 <\delta \leq \epsilon$ such that for all $t \in (0, \delta)$ we have that $\|q_v - q_w\|_C > s_v+s_w$ for all $vw \notin E$.
	We now choose $(p',r') = \alpha(t)$ for some $t \in (0,\delta)$ and note that $(G',p',r')$ is a $C$-packing.
\end{proof}

We are now ready to prove Corollary \ref{cor:generalpos}.

\begin{proof}[Proof of Corollary \ref{cor:generalpos}]
	Define $A_G$ to be the set of regular symmetric bodies in $\mathbb{R}^2$ with the general edge condition for $G$.
	By Lemma \ref{l:genedgesub},
	we may suppose $G$ is a maximal planar graph.	
	By Lemma \ref{l:circlegenedge} we have $\mathbb{D} \in A_G$,
	i.e.~there exists a disc packing $(G,p',r')$ with the general edge condition.
	Let $a,b,c \in V$ be the vertices that define the outer face of $G$ and define $x_a := p'_a$, $x_b := p'_b$ and $x_c := p'_c$.
	If we let $f$ be the continuous map defined in Proposition \ref{p:contmax},
	it follows that $A_G$ is a non-empty open subset of $\mathcal{B}_2$,
	as the set of placements of $G$ with the general edge condition is an open set.
	
	By Proposition \ref{p:phisetdist},
	there exists $N \in \mathbb{N}$ such that $\mathcal{B}_2^\phi(j)$ has a non-empty intersection with $A_G$ for all $j \geq N$.
	Define the map $\Delta :\mathbb{R}^{2|V|} \times \mathbb{R}^{|V|}_{>0} \rightarrow \mathbb{R}$,
	where
	\begin{align*}
		\Delta(p,r) :=
		\prod_{v,w \in V, v \neq w} \det 
		\begin{bmatrix}
				[p_v]_1 & [p_v]_2 \\
				[p_w]_1 & [p_w]_2
		\end{bmatrix}.
	\end{align*}
	By Lemma \ref{l:expgenpos},
	the map $\Delta \circ f \circ \lambda_j$ is an analytic function.
	If we define $D \subset X_j$ to be the set of points $(a,w)$ where $\Delta \circ f \circ \lambda_j (a,w) \neq 0$,
	then either $D$ is an open dense subset of $X_j$ or $D= \emptyset$.
	By Lemma \ref{l:eparam},
	$\lambda_j$ is a continuous surjection.
	Hence,
	either $\lambda_j(D)$ is a dense subset of $\mathcal{B}_2^\phi(j)$ or $\lambda_j(D) = \emptyset$.
	Since $\lambda_j(D) = \mathcal{B}_2^\phi(j) \cap A_G$ then $\mathcal{B}_2^\phi(j) \cap A_G$ is a dense subset of $\mathcal{B}_2^\phi(j)$.
	As this holds for all $j \geq N$ then $A_G$ is dense in $\mathcal{B}_2$ also by Proposition \ref{p:phisetdist}.
\end{proof}

\section{Proof of Theorem \ref{t:symperfect}}\label{sec:symperfect}

We begin with the following definition.

\begin{definition}
	Let $G=(V,E)$ be a planar graph.
	A regular symmetric body $C$ is \emph{$G$-independent} if there exists an independent $C$-packing with contact graph $G$.
\end{definition}

\begin{remark}
	With our new terminology,
	we see that Conjecture \ref{conj:generic} is equivalent to the following statement;
	if $C$ is a regular symmetric body that is not the linear transform of a disc,
	then $C$ is $G$-independent for every $(2,2)$-sparse graph $G$.
\end{remark}

For the entirety of this section we shall be proving the following result.

\begin{lemma}\label{l:symperfect}
	Let $G$ be a $(2,2)$-sparse planar graph.
	Then the set of $G$-independent regular symmetric bodies is an open dense subset of $\mathcal{B}_2$.
\end{lemma}

Using Lemma \ref{l:symperfect},
Theorem \ref{t:symperfect} follows almost immediately.

\begin{proof}[Proof of Theorem \ref{t:symperfect}]
	For each planar graph $G$,
	the set $\mathcal{G}_G$ of $G$-independent bodies is an open dense subset of $\mathcal{B}_2$ by Lemma \ref{l:symperfect}.
	By Proposition \ref{p:baire},
	$\mathcal{B}_2$ is a Baire space,
	hence if we define 
	\begin{align*}
		\mathcal{G} := \bigcap \{ \mathcal{G}_G: G \text{ is $(2,2)$-sparse} \} 
	\end{align*}
	then $\mathcal{G}$ is a comeager subset of $\mathcal{B}_2$.
	As $\mathcal{B}_2$ is a comeager subset of $\mathcal{K}_2$ then $\mathcal{G}$ is a comeager subset of $\mathcal{K}_2$ also.
\end{proof}

We shall prove Lemma \ref{l:symperfect} with two key lemmas;
the first (Lemma \ref{kl:open}) shall prove the set of $G$-independent bodies an open subset of $\mathcal{B}_2$,
and the second (Lemma \ref{kl:dense}) shall prove the set of $G$-independent bodies is dense in $\mathcal{B}_2$.

We will prove Lemma \ref{kl:open} by constructing a continuous map from the neighbourhood of a $G$-independent body into the space of all pairs $(p,r)$.
As the rank of the packing rigidity matrix is lower semi-continuous,
this will prove that every $G$-independent body is contained in an open neighbourhood of $G$-independent bodies.

To prove Lemma \ref{kl:dense},
we devise a method that will allow us to alter the convex body of a given packing so that the placement, radii and contact graph will remain the same but tangents between any two touching convex bodies can be altered.
By doing so we can perturb the entries (and hence change the rank) of the packing rigidity matrix of the original packing by choosing a new regular symmetric body that is sufficiently close to the original.
This will thus allow us to find a $G$-independent body near to any regular symmetric body.

\subsection{The set of \texorpdfstring{$G$}{G}-independent regular symmetric bodies is open}

We will first define the following topological space.
Let $C^1(\mathbb{R}^n)$ be the space of all $C^1$-differentiable maps $f :\mathbb{R}^n \rightarrow \mathbb{R}^n$.
For each compact set $K \subset \mathbb{R}^n$ we define the semi-norm
\begin{align*}
	\|f\|_{C^1,K} := \sup_{x\in K} \|f(x)\| + \sup_{x \in K} \|df(x)\|_{op}.
\end{align*}
With this set of semi-norms,
$C^1(\mathbb{R}^n)$ is a Fr\'{e}chet space.
Now we choose any compact set $D \subset \mathbb{R}^n$ that is the image of $\mathbb{B}^n$ under some diffeomorphism.
We define $C^1(D)$ to be the quotient of $C^1(\mathbb{R}^n)$ by the norm $\|\cdot\|_{C^1,D}$;
alternatively, we can consider $C^1(D)$ to be the set of all maps $f: D \rightarrow \mathbb{R}^n$ which can be extended to some $C^1$-differentiable map $\overline{f} :\mathbb{R}^n \rightarrow \mathbb{R}^n$ with $\overline{f}(x)=f(x)$ for all $x \in D$.
When gifted the norm $\|f\|_{C^1}:=\|f\|_{C_1,D}$,
the space $C^1(D)$ is a Banach space.

\begin{remark}
	By the Whitney extension theorem (see \cite{whitney}),
	if $U$ is an open set containing $D$ and $f:U \rightarrow \mathbb{R}^n$ is $C^1$-differentiable then $f|_D \in C^1(D)$.
\end{remark}

We shall extend the definition of local $C^1$-diffeomorphism to $C^1(D)$;
namely, a map $f \in C^1(D)$ is a \emph{local $C^1$-diffeomorphism} if $df(x)$ is invertible for all $x \in D$.
By the inverse function theorem (see \cite[Theorem 2.5.2]{manifold}),
if $f$ is an injective local $C^1$-diffeomorphism then $f^{-1} \in C^1(f(D))$ is also an injective local $C^1$-diffeomorphism with $df^{-1}(f(x)) = df(x)^{-1}$ for all $x \in D$.

\begin{lemma}\label{l:difflip}
	Let $f \in C^1(D)$ be an injective local $C^1$-diffeomorphism.
	Then for all $x,y \in D$,
	\begin{align*}
		\frac{1}{\sup_{z\in D} \|df(z)^{-1}\|_{op}} \|x-y\| \leq \|f(x)-f(y)\| \leq \sup_{z\in D} \|df(z)\|_{op} \|x-y\|.
	\end{align*}
\end{lemma}

\begin{proof}
	It follows from the mean value theorem that 
	\begin{align}\label{eq:difflip1}
		\|f(x)-f(y)\| \leq \sup_{z \in D} \|df(z)\|_{op} \|x-y\|
	\end{align}
	for all $x,y \in D$.
	Since $f^{-1}$ is well-defined on the domain $f(D)$ then we similarly note that
	\begin{align}\label{eq:difflip2}
		\|x-y\|=\|f^{-1}(f(x)))-f^{-1}(f(y))\| \leq \sup_{z \in D} \|df^{-1}(f(z))\|_{op} \|f(x)-f(y)\|
	\end{align}
	for all $x,y \in D$.
	As $d f^{-1}(f(x)) = df(x)^{-1}$ for all $x \in D$,
	we can combine (\ref{eq:difflip1}) and (\ref{eq:difflip2}) to arrive at the desired result.
\end{proof}

\begin{lemma}\label{l:immopen}
	The set of local $C^1$-diffeomorphisms in $C^1(D)$ is an open set.
	Further,
	the maps $f \mapsto \sup_{z \in D} \|df(z)\|_{op}$ and $f \mapsto \sup_{z \in D} \|df(z)^{-1}\|_{op}$ are both continuous on the set of local $C^1$-diffeomorphisms.
\end{lemma}

\begin{proof}
	Define $M_{n}$ to be the space of $n \times n$ matrices and $GL_{n}$ to be open dense subset of invertible matrices.
	Choose any local $C^1$-diffeomorphism $g \in C^1(D)$.
	As the map $df: x \mapsto df(x)$ is a continuous map,
	then the set $df(D)$ is a compact subset of $GL_{n}$.
	Hence there exists $\epsilon >0$ such that,
	given $M \in M_{n}$,
	if $\|M - df(x)\|_{op} < \epsilon$ for some $x \in D$ then $M \in GL_{n}$.
	It follows that if $\|f - g\|_{C^1} < \epsilon$ then $f$ is a local $C^1$-diffeomorphism also.
	
	Define the Banach space $X := \{ df : f \in C^1(D) \}$ with the norm
	\begin{align*}
		\|df\|_X:= \sup_{z \in D} \|df(z)\|_{op}
	\end{align*}
	As the map $f  \mapsto df$ is continuous with respect to either $\|\cdot \|_{C^1}$ or $\|\cdot\|_X$,
	then $f \mapsto \|df(z)\|_X$ is continuous.
	As the inversion map $\sigma : M \mapsto M^{-1}$ is continuous on $GL_{n}$,
	then the map $df  \mapsto \sigma \circ df$ is continuous on the open set of local $C^1$-diffeomorphisms;
	hence $f \mapsto \|df(z)^{-1}\|_X$ is continuous also.	
\end{proof}

While not all local $C^1$-diffeomorphisms will be injective,
we can always restrict the domain to form an injective map by the inverse function theorem.
Later we shall require some control over how much we need to restrict our domain by.
To do so we first need the following result.

\begin{theorem}\cite[Theorem 4]{MusRog07}\label{t:MusRog07}
	Let $U \subset \mathbb{R}^n$ be an open subset that contains the closed ball $R \mathbb{B}^n$,
	and let $f :U \rightarrow \mathbb{R}^n$ be a local $C^1$-diffeomorphism.
	Then $f|_{M_1^{-1} M_2 \mathbb{B}^n}$ is injective,
	where 
	\begin{align*}
		M_1 := \sup_{\|z\| \leq R} \|df(z)\|_{op}, \qquad M_2 := \int_0^R \frac{1}{\sup_{\|z\| \leq u} \|df(z)^{-1}\|_{op}} d u.
	\end{align*}
\end{theorem}

We now can simplify Theorem \ref{t:MusRog07} to the following.

\begin{lemma}\label{l:imminv}
	Let $f \in C^1(R\mathbb{B}^n)$ be a local $C^1$-diffeomorphism for some $R >0$.
	If 
	\begin{align*}
		r < \frac{R}{\left( \sup_{\|z\| \leq R} \|df(z)\|_{op} \right) \left(\sup_{\|z\| \leq R} \|df(z)^{-1}\|_{op} \right) }.
	\end{align*}
	then $f|_{r\mathbb{B}^n}$ is injective.
\end{lemma}

\begin{proof}
	Any map in $C^1(R\mathbb{B}^n)$ can be extended to a $C^1$-differentiable map on $U=\mathbb{R}^n$.
	From Theorem \ref{t:MusRog07} we only need to show that if $r$ is as stated then $r \leq M_1^{-1} M_2$.
	This now follows as
	\begin{align*}
		M_2 = \int_0^R \frac{1}{\sup_{\|z\| \leq u} \|df(z)^{-1}\|_{op}} d u \geq \int_0^R \frac{1}{\sup_{\|z\| \leq R} \|df(z)^{-1}\|_{op}} d u = \frac{R}{\sup_{\|z\| \leq R} \|df(z)^{-1}\|_{op}}.
	\end{align*}
\end{proof}

\begin{remark}
	For any invertible $n \times n$ matrix $M$,
	we have
	$1 = \|M M^{-1}\|_{op} \leq \|M \|_{op}\| M^{-1}\|_{op}$.
	It follows that the value $r >0$ given in Lemma \ref{l:imminv} is strictly less than $R$,
	hence $r \mathbb{B}^n \subset R \mathbb{B}^n$.
\end{remark}

\begin{lemma}\label{l:bilip}
	Suppose $g\in C^1(R\mathbb{B}^n)$ is an injective local $C^1$-diffeomorphism.
	Then there exists $c,\epsilon >0$ and $0<r<R$ such that for all $f\in C^1(R\mathbb{B}^d)$ with $\|f - g\|_{C^1} < \epsilon$,
	we have that $f|_{r\mathbb{B}^n}$ is an injective local $C^1$-diffeomorphism with
	\begin{align*}
		\frac{1}{c}\|x-y \| \leq \|f(x) - f(y) \| \leq c\|x - y\|
	\end{align*}
	for all $x,y \in r \mathbb{B}^n$.
\end{lemma}

\begin{proof}
	By Lemma \ref{l:immopen},
	there exists $c_1,c_2,\epsilon >0$ such that for all $\|f - g\|_{C^1} < \epsilon$ we have that $df(x)$ is invertible for each $x \in R\mathbb{B}^n$ and 
	\begin{align*}
		c_1 \leq \sup_{\|z\| \leq R} \|df(z)^{-1}\|_{op} \leq c_2 ,\qquad c_1\leq  \sup_{\|z\| \leq R} \|df(z)\|_{op} \leq c_2.
	\end{align*}
	Choose $r < (1/c_1)^2$.
	By Lemma \ref{l:imminv},
	if $\|f - g\|_{C^1} < \epsilon$ then $f|_{r\mathbb{B}^n}$ is injective.
	If we set $c =c_2$,
	then the required result holds by Lemma \ref{l:difflip}.
\end{proof}

We are now ready for the following generalisation of the implicit function theorem.

\begin{theorem}\label{t:implicit}
	Let $(S,d)$ be a metric space,
	$X$ be a $n$-dimensional affine space,
	$X' \subset X$ be an open set,
	$g: S \times X' \rightarrow \mathbb{R}^n$ be a continuous function,
	and $(s,t) \in S \times X'$ be a point where $g(s,t)=0$.
	Suppose that there exists open sets $A \subset S$ and $B \subset X'$ with $s \in A$ and $t \in B$ so that for every $a \in A$ the map 
	\begin{align*}
		g_a: X' \rightarrow \mathbb{R}^n, ~ b \mapsto g(a,b)
	\end{align*}
	is $C^1$-differentiable on $B$ with invertible derivative at every point of $B$,
	and the map $(a,b) \mapsto d g_a (b)$ is continuous on $A \times B$.
	Then there exists an open neighbourhood $U \subset A$ of $s$ and a unique continuous map $f: U \rightarrow B$ so that $f(s)=t$ and $g(a,f(a))=0$.
\end{theorem}

\begin{proof}
	By translating both $X'$ and $X$, we may assume $t=0$.
	By the inverse function theorem,
	the map $g_a|_B$ is a local $C^1$-diffeomorphism for each $a \in A$.
	Choose $R>0$ such that $R \mathbb{B}^n \subset B$ and $g_s|_{R\mathbb{B}^n}$ is injective.
	Since $g$ and $(a,b) \mapsto d g_a (b)$ are continuous on $A \times B$,
	then the map 
	\begin{align*}
		A \rightarrow C^1(R \mathbb{B}^n) , ~ a \mapsto g_a|_{R\mathbb{B}^n}
	\end{align*}
	is also continuous.
	Hence by Lemma \ref{l:bilip},
	there exists $c>0$ and open neighbourhoods $A' \subset A$, $B' \subset R \mathbb{B}^n$ of $s,t$ respectively such that 
	\begin{align*}
		\frac{1}{c}\|x-y \| \leq \|g_a(x) - g_a(y) \| \leq c\|x - y\|
	\end{align*}
	for all $a \in A'$ and $x,y \in B'$.
	
	Define the map
	\begin{align*}
		\Psi : A' \times B' \rightarrow A' \times \mathbb{R}^n, ~ (a,b) \mapsto (a, g(a,b)) = (a,g_a(b)).
	\end{align*}
	If we define $d'$ to be the metric on $A' \times \mathbb{R}^n$ with $d'((a,x),(a',x')) := d(a,a') + \|x-x'\|$,
	and likewise define $d''$ to be the metric on $A' \times B'$ with $d''((a,b),(a',b')) := d(a,a') + \|b-b'\|$,
	then
	\begin{align*}
		\frac{1}{\max\{ 1,c\}} d''((a,b),(a',b')) \leq d'(\Psi(a,b), \Psi(a',b')) \leq \max\{ 1,c\} d''((a,b),(a',b')).
	\end{align*}
	Hence, $\Psi$ is a homeomorphism onto it's image and $\Psi(A' \times B')$ is an open set.
	Define 
	\begin{align*}
		\Gamma : \Psi(A' \times B') \rightarrow A' \times B', ~ x \mapsto (\Gamma_A(x),\Gamma_B(x))
	\end{align*}
	to be the inverse of $\Psi$,
	with $\Gamma_A(x) \in A'$ and $\Gamma_B(x) \in B'$.
	As $\Psi(A' \times B')$ is an open set that contains $(s,0) = (s, g(s,t))$,
	there exists an open neighbourhood $U \subset A'$ of $s$ where $(a,0) \in \Psi(A' \times B')$ for all $a \in U$.
	We now note that
	\begin{align*}
		(a,0) = \Psi \circ \Gamma (a, 0) = (  \Gamma_A(a, 0), g(\Gamma_A(a, 0),\Gamma_B(a, 0)) )
	\end{align*}
	 for all $a \in U$,
	 hence $\Gamma_A(a, 0) = a$ and $g(a,\Gamma_B(a, 0)) = 0$.
	We now set $f (a) := \Gamma_B(a, 0)$ for every $a \in U$.	
\end{proof}

We are now ready to prove our first key lemma of the section.
We note that we may drop the requirement that $G$ is $(2,2)$-sparse;
this is as if $G$ is not $(2,2)$-sparse then the set of $G$-independent bodies can be seen to be empty (and hence open) by applying Theorem \ref{t:laman}.

\begin{lemma}\label{kl:open}
	For any planar graph $G=(V,E)$,
	the set of $G$-independent regular symmetric bodies is an open subset of $\mathcal{B}_2$.
\end{lemma}

\begin{proof}
	Suppose there exists $G$-independent regular symmetric body $C'$ with independent $C'$-packing $P=(G,p',r')$.
	For the matrix $R_{C'}(G,p,r)$,
	let $c(v,i)$ be the $i$-th coordinate point column of $v$ and $c(v)$ be the radii column of $v$ (see Section \ref{sec:packmaps}).
	As $R_{C'}(G,p',r')$ has independent rows then there exists $|E|$ independent columns $I \subset \{ c(v,i),c(v) : i \in \{1,2\}, v \in V\}$.
	For any $|E| \times 3|V|$ real valued matrix $M$, 
	define the $|E| \times |E|$ matrix $M|_I$ by deleting any column not in $I$;
	by our choice of columns we have $R_{C'}(G,p',r')|_I$ is invertible.
	
	Define $X$ to be the affine subspace of $\mathbb{R}^{2|V|} \times \mathbb{R}^{|V|}$ of pairs $(p,r)$ where for any $c(v,i) \notin I$ we have $[p_v]_i=[p'_v]_i$ and for any $c(v) \notin I$ we have $r_v=r'_v$,
	and define the open subset $X' := X \cap \mathbb{R}^{2|V|} \times \mathbb{R}^{|V|}_{>0}$ of $X$.
	By Proposition \ref{p:phicont},
	there exist open neighbourhoods $U_\mathcal{B} \subset \mathcal{B}_2$ and $U_{X'} \subset X'$ of $C'$ and $(p',r')$ respectively where $R_C(G,p,r)|_I$ is invertible for all $C \in U_\mathcal{B}$ and $(p,r) \in U_{X'}$.	
	Now define the map 
	\begin{align*}
		g : \mathcal{B}_2 \times X' \rightarrow \mathbb{R}^{|E|}, ~ (C,(p,r)) \mapsto h_{G,C}(p,r).
	\end{align*}
	It is immediate that $g(C',(p',r'))=0$; also, $g$ is continuous by Proposition \ref{prop:hausvsdash}.
	For any $C \in \mathcal{B}_2$,
	the map
	\begin{align*}
		g_C:  X' \rightarrow \mathbb{R}^n, ~ (p,r) \mapsto g(C,(p,r)) = h_{G,C}(p,r)
	\end{align*}
	is $C^1$-differentiable by Proposition \ref{p:support},
	and the derivative of $g_c$ at $(p,r)$ is $R_C(G,p,r)|_I$.
	As the map $(C,x) \mapsto \varphi_C(x)$ is continuous on $\mathcal{B}_2 \times \mathbb{R}^2$ by Proposition \ref{p:phicont},
	the map $(C,(p,r)) \mapsto dg_C(p,r)$ is also continuous.
	By Theorem \ref{t:implicit},
	there exists an open subset $U_{\mathcal{B}_2}' \subset U_{\mathcal{B}_2}$ and a unique continuous function $f : U_{\mathcal{B}_2}' \rightarrow U_{X'}$ such that $f(C') = (p',r')$ and $g(C,f(C)) = 0$ for all $C \in U_{\mathcal{B}_2}'$;
	further,
	we may choose $U_{\mathcal{B}_2}'$ to be sufficiently small so that for $C \in U'_{\mathcal{B}_2}$ and $(p,r)=f(C)$ we have $\|p_v-p_w\|_C > r_v+r_w$ for all $vw \notin E$.
	It follows that given $C \in U_{\mathcal{B}_2}'$ and $(p,r) = f(C)$, the triple $(G,p,r)$ is an independent $C$-packing.
	Hence $U_{\mathcal{B}_2}'$ is an open neighbourhood of $C'$ of $G$-independent regular symmetric bodies.
\end{proof}

\subsection{The set of \texorpdfstring{$G$}{G}-independent regular symmetric bodies is dense}\label{sec:kldense}

We remember that $\mathbb{T} := (-\pi,\pi]$ has the topology generated by the metric $|a - b|_{\sim} := \min \{ |a-b|, 2\pi - |a - b| \}$.
For any closed connected subset $I \subseteq \mathbb{T}$,
we denote by $C(I,\mathbb{R})$ the Banach space of all functions $f :I \rightarrow \mathbb{R}$ with the supremum norm $\|f\|_I := \sup_{t \in I} | f(t)|$.
We now define $\mathcal{F}_I \subset C(I,\mathbb{R})$ to be the set of all $C^2$-differentiable functions $f : I \rightarrow \mathbb{R}_{>0}$ where $f(t \pm \pi) = f(t)$ (if $t \pm \pi \in I$) and
\begin{align*}
	c_f(t) := f(t)^2 + 2(f'(t))^2 - f(t)f''(t) >0
\end{align*}	
for all $t \in I$.

\begin{lemma}\label{l:funisom}
	There exists a well-defined homeomorphism $F: \mathcal{B}_2^+ \rightarrow \mathcal{F}_{\mathbb{T}}$ where
	\begin{align*}
		F_C(t) := \frac{1}{\| (\cos t , \sin t )\|_C}
	\end{align*}
	for each $C \in \mathcal{B}^+_2$ and $t \in \mathbb{T}$.
\end{lemma}

\begin{proof}
	Choose any $C \in \mathcal{B}_2^+$.
	As $\|\cdot\|_C$ is $C^2$-differentiable, centrally symmetric and has a non-empty interior that contains $0$,
	then it is immediate that $F(C)$ is also $C^2$-differentiable, positive and $F_C(t \pm \pi) = F_C(t)$ for all $t \in \mathbb{T}$.
	Define the $C^2$-differentiable curve $\alpha :\mathbb{T} \rightarrow \mathbb{R}^2 \times \{0\}$ where $\alpha(t) := (F_C(t)\cos t,F_C(t)\sin t,0)$.
	We immediately note that $\alpha'(t) \times \alpha''(t) = (0,0,c_{F_C}(t))$.
	As $C$ is convex then $\alpha$ is always ``turning left'', i.e.~the third coordinate of $\alpha'(t) \times \alpha''(t)$ must be non-negative.
	It follows that $C$ has positive curvature (i.e.~$\alpha'(t), \alpha''(t)$ are linearly independent for all $t \in \mathbb{T}$) if and only if $c_{F_C}(t) >0$ for all $t \in \mathbb{T}$,
	hence $F_C \in \mathcal{F}_\mathbb{T}$.
	
	Choose any $f \in \mathcal{F}_\mathbb{T}$ and define the set
	\begin{align*}
		C:= \{ (r \sin t, r \cos t) : t \in \mathbb{T}, r \leq f(t) \}.
	\end{align*}
	As $f(t \pm \pi) = f(t)$ for all $t \in \mathbb{T}$ and $f$ is positive,
	$C$ is centrally symmetric with non-empty interior.
	If we define the $C^2$-differentiable curve $\alpha :\mathbb{T} \rightarrow \mathbb{R}^2 \times \{0\}$ with $\alpha(t) := (f(t)\cos t,f(t)\sin t,0)$,
	then as the third coordinate of $\alpha'(t) \times \alpha''(t)$ is $c_f(t) >0$ we have that $\alpha$ is always ``turning left'';
	it follows that $C$ must be convex and thus $C \in \mathcal{K}_2$.
	As $\alpha'(t), \alpha''(t)$ are linearly independent for all $t \in \mathbb{T}$ then $C \in \mathcal{B}^+_2$ by Lemma \ref{l:schneider}.
	It now follows that $F$ is bijective, since $F$ is clearly injective also.
	
	By Proposition \ref{prop:hausvsdash},
	the map $F_C$ is continuous.
	Choose any convergent sequence $(f_n)_{n \in \mathbb{N}}$ in $\mathcal{F}_\mathbb{T}$ with limit $f \in \mathcal{F}_{\mathbb{T}}$,
	and define $C_n := F^{-1}(f_n)$, $C := F^{-1}(f)$.
	Given $s(t) := (\cos t, \sin t)$,
	we now note that
	\begin{align*}
		d_H(C,C_n) \leq \sup_{t \in \mathbb{T}} \| f(t) s(t) - f_n(t)s(t)\| = d_{\mathbb{T}}(f,f_n) \rightarrow 0
	\end{align*}
	as $n \rightarrow \infty$.
	Hence $F$ is a homeomorphism as required.
\end{proof}

\begin{lemma}\label{l:raderiv1}
	Let $C \in \mathcal{B}_2^+$ and $f := F_C$.
	Then given $s(t) := (\cos t, \sin t)$, we have
	\begin{align*}
		\varphi_C( s(t)) = f(t)^{-2} s(t) - f'(t) f(t)^{-3} s'(t).
	\end{align*}	
	for all $t\in \mathbb{T}$.
\end{lemma}

\begin{proof}
	By differentiating we have
	\begin{align} \label{l:raderiv11}
		f'(t) = \frac{d}{dt} \frac{1}{\|s(t) \|_C} = \frac{-\varphi_C(s(t)/\|s(t)\|_C) . s'(t)}{ \|s(t) \|_C^{2}} = -(\varphi_C(s(t)) . s'(t)) f(t)^3.
	\end{align}
	As $\varphi_C(s(t))$ is the support of $s(t)$ then
	\begin{align} \label{l:raderiv12}
		\varphi_C(s(t)) . s(t) = \|s(t)\|_C^2 = f(t)^{-2}.
	\end{align}
	As $\{s(t),s'(t)\}$ is an orthonormal basis of $\mathbb{R}^2$ for each $t \in \mathbb{T}$ then 
	\begin{align} \label{l:raderiv13}
		\varphi_C(s(t)) = \left(\varphi(s(t)). s(t) \right) s(t) + \left(\varphi(s(t)). s'(t) \right) s'(t).
	\end{align}
	We now reach our desired equality by equations (\ref{l:raderiv11}), (\ref{l:raderiv12}) and (\ref{l:raderiv13}).
\end{proof}

The following results will enable us to alter the tangents of a packing without breaking any contacts.
This will allow us later to approximate convex bodies with $G$-independent regular symmetric bodies while maintaining the structure of the original packing.

\begin{lemma}\label{l:dense1}
	Let $I$ be a closed interval in $\mathbb{T}$ with end points $x_1,x_2$,
	and suppose $|x_1-x_2|_{\sim} < \pi$.
	Fix any	$f \in \mathcal{F}_I$, $c \in I \setminus \{x_1,x_2\}$ and $\epsilon >0$.
	Then there exists $\delta >0$ such that for all $a \in \mathbb{R}$ where $|a| < \delta$,
	there exists $g \in \mathcal{F}_I$ where
	\begin{enumerate}[(i)]
		\item \label{l:dense1item1} $\|f-g\|_I <\epsilon$,
		\item \label{l:dense1item2} $g(c)= f(c)$ and $g'(c)=f'(c)+a$, and
		\item \label{l:dense1item4} $g(x_i)= f(x_i)$, $g'(x_i)= f'(x_i)$, $g''(x_i) = f''(x_i)$ for all $i \in \{1,2\}$.
	\end{enumerate}
\end{lemma}

\begin{proof}
	By altering the domains of maps in $\mathcal{F}_I$, we may assume $x_1=0$ and $x_2=1$.
	Define the polynomial
	\begin{align*}
		p (x) := x^3(x-1)^3 (x-c).
	\end{align*}
	Then 
	\begin{eqnarray*}
		p' (x) &=& 3x^2(x-1)^3 (x-c) + 3x^3(x-1)^2 (x-c)  + x^3(x-1)^3, \\
		p'' (x) &=& 18 x^2(x-1)^2(x-c) + 6x(x-1)^3(x-c) +6x^3(x-1)(x-c)\\
		&~& + \, 6x^2(x-1)^3 + 6x^3(x-1)^2,
	\end{eqnarray*}
	and 
	\begin{align*}
		p(0)=p(1)=p(c)=0, \quad p'(0)=p'(1)=0, \quad p''(0)=p''(1)=0,\quad p'(c) = c^3(c-1)^3.
	\end{align*}
	Set $\lambda_a :=\frac{a}{c^3(c-1)^3}$ and define $g_a := f +  \lambda_a p$ for each $a \in \mathbb{R}$.
	It is immediate that $g_a$ is a $C^2$-differentiable function that satisfies the properties (\ref{l:dense1item2}) and (\ref{l:dense1item4}).
	Define $C^2(I,\mathbb{R}) \subset C(I,\mathbb{R})$ to be the subset of $C^2$-differentiable functions with the subspace topology.
	Then the map $a \mapsto g_a$ is a continuous curve in $C^2(I,\mathbb{R})$,
	hence there exists $\delta_0>0$ so that $\|f-g_a\|_I< \epsilon$ for all $|a|< \delta_0$.
	
	Define the functions $k_1,k_2 : C^2(I,\mathbb{R}) \rightarrow \mathbb{R}$,
	where 
	\begin{align*}
		k_1(h) := \inf_{t \in I} h(t), \quad k_2(h) := \inf_{t \in I} h(t)^2 + 2(h'(t))^2 - h(t)h''(t).
	\end{align*}
	The function $k_1$ is continuous on $C^2(I,\mathbb{R})$,
	hence there exists $\delta_1>0$ such that $k_1(g_a) >0$ for all $|a|< \delta_1$.
	For any $a \in \mathbb{R}$,
	we have 	
	\begin{eqnarray*}
		k_2(g_a) 
		&=& \inf_{t \in I} \left( (f(t)+\lambda_a p(t))^2 +2 (f'(t)+\lambda_a p'(t))^2 - (f(t)+\lambda_a p(t))(f''(t)+\lambda_a p''(t)) \right) \\
		&\geq& k_2(f) + \lambda_a^2 k_2(p) + \lambda_a \left( \inf_{t \in I} \left( 2f(t)p(t) + 4 f'(t)p'(t) - f(t)p''(t) - p(t)f''(t) \right) \right).
	\end{eqnarray*}
	As $k_2(f)>0$ and the map $a \mapsto \lambda_a$ is continuous,
	there exists $\delta_2>0$ such that $k_2(g_a)>0$ for all $|a|< \delta_2$.
	If $|a| <\min\{\delta_1,\delta_2\}$ then $g_a \in \mathcal{F}_I$,
	hence the result holds with $\delta := \min\{\delta_0,\delta_1,\delta_2\}$.
\end{proof}

\begin{lemma}\label{l:dense2}
	Let $C \in \mathcal{B}_2^+$ and $x_1,\ldots,x_n \in \partial C$ be pairwise linearly independent,
	and choose any $\epsilon >0$.
	Then there exists $\delta >0$ so that for any set $y_1,\ldots, y_n \in \mathbb{R}^2 \setminus \{0\}$ where $\|\varphi_C(x_i) - y_i\| < \delta$ for all $i=1,\ldots,n$, the following holds;
	there exists $C' \in \mathcal{B}^+_2$ where $d_H(C,C')<\epsilon$,
	$\varphi_{C'} (x_i)$ is a scalar multiple of $y_i$
	and $\|x_i\|_{C'} =1$ for all $i=1,\ldots,n$.
\end{lemma}

\begin{proof}
	Let $F$ be the homeomorphism from Lemma \ref{l:funisom}, $f := F^{-1}(C)$ and $s(t) := (\cos t, \sin t)$.
	Choose $t_1, \ldots, t_n \in \mathbb{T}$ so that $x_i = f(t_i)s(t_i)$ for each $i\in \{1,\ldots,n\}$.
	Then there exists $\epsilon'>0$ such that if $\|g-f\|_I< \epsilon'$ then $d_H(F^{-1}(g),C) <\epsilon$.
	For each $i\in \{1,\ldots,n\}$,
	choose a closed interval $I_i :=[w^1_i,w^2_i] \subset \mathbb{T}$ so that $|w^1_i-w^2_i|_{\sim}<\pi$,
	$I_i \cap I_j = \emptyset$ and $I_i \cap -I_j = \emptyset$ for all $i \neq j$.
	By Lemma \ref{l:dense1},
	there exists $\delta' >0$ such that for every $i \in \{1,\ldots,n\}$ and $|a_i |<\delta'$,
	there exists a function $g_i :I_i \rightarrow \mathbb{R}$ where 
	\begin{enumerate}[(i)]
		\item $\|f-g_i\|_{I_i} <\epsilon'$,
		\item $g_i(t_i)= f(t_i)$ and $g'_i(t_i)= f'(t_i)+ a_i$, and
		\item $g(w^j_i)= f(w^j_i)$, $g'(w^j_i)= f'(w^j_i)$, $g''(w^j_i) = f''(w^j_i)$ for all $j\in \{1,2\}$.
	\end{enumerate}
	We now define $g \in \mathcal{F}_{\mathbb{T}}$ by
	\begin{align*}
		g(x) :=
		\begin{cases}
			g_i(x) &\text{if } x \in I_i, \\
			g_i(x) &\text{if } x \in -I_i, \\
			f(x) &\text{otherwise}.
		\end{cases}
	\end{align*}
	If $C'\in \mathcal{B}_2^+$ is the unique element where $F(C')=g$,
	then $d_H(C,C')< \epsilon$, $\|x_i\|_{C'} = 1$ and by Lemma \ref{l:raderiv1} we have for each $i \in \{1,\ldots,n\}$,
	\begin{eqnarray*}
		\varphi_{C'}(x_i) &=& g(t_i)^{-2} s(t_i) - g(t_i)^{-3} g'(t_i) s'(t_i) \\
		&=& f(t_i)^{-2} s(t_i) - f(t_i)^{-3} f'(t_i) s'(t_i) - f(t_i)^{-3} a_i s'(t_i) \\
		&=& \varphi_C(x_i) - f(t_i)^{-3} a_i s'(t_i).
	\end{eqnarray*}
	For each $i \in \{ 1,\ldots, n\}$,
	define 
	\begin{align*}
		Y_i := \left\{ \varphi_C(x_i) +b_i s'(t_i) :  |b_i f(t')^3)| < \delta' \right\}.
	\end{align*}
	By Lemma \ref{l:raderiv1},
	$\varphi_C(x_i)$ and $s'(t_i)$ must be linearly independent.
	Hence there exists $\delta >0$ such that for any $i \in \{1,\ldots,n\}$,
	if $\|\varphi_C(x_i) - y_i\|<\delta$ then $\lambda_i y_i \in Y_i$ for some $\lambda_i >0$.

	Choose $y_1, \ldots, y_n \in \mathbb{R}^2 \setminus \{0\}$ where $\|\varphi_C(x_i) - y_i\|<\delta$,
	and let $\lambda_1,\ldots, \lambda_n$ be the positive scalars where $\lambda_i y_i \in Y_i$ for each $i \in \{1,\ldots,n\}$.
	Then as previously shown,
	there exists $C'\in \mathcal{B}_2^+$ such that $\varphi_{C'}(x_i) = \lambda_i y_i$ for each $i \in \{1,\ldots,n\}$,
	$\|x_i\|_{C'}=1$ and $d_H(C,C')<\epsilon$ as required.
\end{proof}

The following lemma is a slight rewording of \cite[Theorem 2.18]{WhiteWhiteley87} so as to avoid having to use the language of \emph{$k$-frames};
a generalisation of rigidity matrices.

\begin{lemma}\label{l:dense3}
	Let $G=(V,E)$ be a finite graph and define $M_G$ to be the set of all $|E|\times d|V|$ real-valued matrices with entries $a_{vw,(u,i)}$ for each $vw \in E$ and $(u,i) \in V \times \{1,\ldots,d\}$,
	where $a_{vw,(v,i)} = -a_{vw,(w,i)}$ and $a_{vw,(u,i)} = 0$ if $u \notin \{v,w\}$.
	Then the set
	\begin{align*}
		M_G' := \{ M \in M_{G} : \rank M = |E| \}
	\end{align*}
	is an open dense subset of $M_{G}$ if $G$ is $(d,d)$-sparse,
	and $M_G'=\emptyset$ otherwise.
\end{lemma}

We are now ready for our final key lemma.
 
\begin{lemma}\label{kl:dense}
	For any planar $(2,2)$-sparse graph $G=(V,E)$,
	the set of $G$-independent regular symmetric bodies is a dense subset of $\mathcal{B}_2$.
\end{lemma}

\begin{proof}
	Choose any $C \in \mathcal{B}_2$ and fix some $\epsilon >0$.
	By Corollary \ref{cor:generalpos},
	there exists $C' \in \mathcal{B}_2$ with the general edge condition such that $d_H(C,C') < \epsilon/2$;
	as $\mathcal{B}_2^+$ is dense in $\mathcal{B}_2$ (Proposition \ref{p:phisetdist}),
	we may assume $C' \in \mathcal{B}_2^+$ also.
	
	Let $P=(G,p,r)$ be a $C'$-packing in general edge position and order the vertices of $V$.
	For each $vw \in E$ with $v<w$,
	define $x_{vw} := \frac{p_v-p_w}{\|p_v-p_w\|_{C'}} \in \partial C'$.
	As $P$ is in general edge position,
	the pair $x_{vw},x_{v'w'}$ are linearly independent for every distinct pair $vw,v'w'\in E$.
	By Lemmas \ref{l:dense2} and \ref{l:dense3} applied to $C'$ and $\{x_{vw} : vw \in E\}$,
	there exists $C'' \in \mathcal{B}_+^2$ so that $\| x_{vw}\|_{C'}=1$ for each $vw \in E$,
	$d_H(C',C'')< \epsilon/2$ and $\rank R_{C'}(G,p) = |E|$;
	we may also assume that $\|p_v-p_w\|_{C''} > r_v + r_w$ for all $vw \notin E$ by Proposition \ref{prop:hausvsdash}.
	For any $vw \in E(G)$ we have
	\begin{align*}
		\|p_v-p_w\|_{C''} = \|p_v-p_w\|_{C'} = r_v + r_w,
	\end{align*}
	hence $(G,p,r)$ is an independent $C''$-packing.
	Finally, 
	\begin{align*}
		d_H(C,C'') \leq d_H(C,C') + d_H(C',C'') < \epsilon.
	\end{align*}
\end{proof}

\begin{finalremark}
	Although we have been restricting to the case of central symmetric convex bodies,
	we can also consider what results we could gain when this assumption is dropped.
	There are, unfortunately, a few immediate problems with this.
	For a pair of touching homothetic copies $r_1C+p_1$ and $r_2 C+p_2$ of a centrally symmetric convex body $C$,
	the orientation of their separating hyperplane is determined solely by the vector $p_1-p_2$.
	If $C$ is not centrally symmetric then this may not remain true, however,
	and hence we cannot obtain all the required information we need from the Minkowski functionals alone.
	An alternative function we could use to determine the constraints of a packing would be the support functions of the individual convex bodies,
	although the orientation of the separating hyperplane would need to be kept in consideration for all calculations.
	With this alternative method we would conjecture that similar results to ours could be obtained for homothetic packings of general convex bodies,
	and possibly also for packings that allow for rotations.
\end{finalremark}

\end{document}